\documentclass[twocolumn, 10pt]{IEEEtran}
\usepackage{graphicx,amsmath,amssymb,amsthm,amsfonts,bm}
\usepackage[noadjust]{cite}
\usepackage{braket} 		%For command \Set

\allowdisplaybreaks

\theoremstyle{plain}

\newtheorem{defn}{Definition}
\newtheorem{lem}{Lemma}

\newtheorem{thm}{Theorem}

\newcommand{\abs}[1]{\left| #1 \right|}

\newcommand{\expect}[1]{\mathbb{E}\left[#1\right]}

\newcommand{\paren}[1]{\left( #1 \right)}

\newcommand{\Z}{\mathbb{Z}}

\newcommand{\wt}[1]{\widetilde{#1}}
\newcommand{\argmin}{\operatorname{argmin}}

\newcommand{\bZ}{\bm{Z}}
\newcommand{\bY}{\bm{Y}}
\newcommand{\dd}{\mathrm{d}}

\newcommand{\Pav}{\overline{P}}

\newcommand{\pmax}{P_{\text{max}}}
\newcommand{\pmin}{P_{\text{min}}}

\newcommand{\oW}{\overline{W}}
\newcommand{\pconst}{P_{\text{const}}}
\newcommand{\rmax}{R^{\text{max}}}
\newcommand{\pirand}{\pi_{\text{rand}}}

\begin{document}

\title{Delay and Power-Optimal Control in Multi-Class Queueing Systems}

\author{\large{Chih-ping~Li,~\IEEEmembership{Student Member,~IEEE} and Michael~J.~Neely,~\IEEEmembership{Senior Member,~IEEE}}%
\thanks{Chih-ping Li (web: http://www-scf.usc.edu/$\sim$chihpinl) and Michael J. Neely (web: http://www-rcf.usc.edu/$\sim$mjneely) are with the Department of Electrical Engineering, University of Southern California, Los Angeles, CA 90089, USA.}    %
\thanks{This material is supported in part  by one or more of the following: the NSF Career grant CCF-0747525, the Network Science Collaborative Technology Alliance sponsored by the U.S. Army Research Laboratory.}%
}

%\numberofauthors{2}
%
%\author{
%\alignauthor Chih-ping Li\\
%       \affaddr{Electrical Engineering Department}\\
%       \affaddr{University of Southern California}\\
%       \affaddr{3740 McClintock Ave} \\
%       \affaddr{Los Angeles, CA, 90089, USA}\\
%       \email{chihpinl@usc.edu}
%\alignauthor Michael J. Neely\\
%       \affaddr{Electrical Engineering Department}\\
%       \affaddr{University of Southern California}\\
%       \affaddr{3740 McClintock Ave} \\
%       \affaddr{Los Angeles, CA, 90089, USA}\\
%       \email{mjneely@usc.edu}
%}

%\author{
%\alignauthor \hfill \\
%       \affaddr{ \hfill}\\
%       \affaddr{ \hfill}\\
%       \affaddr{ \hfill} \\
%       \affaddr{ \hfill}\\
%       \email{ \hfill}
%\alignauthor  \hfill \\
%       \affaddr{ \hfill }\\
%       \affaddr{ \hfill}\\
%       \affaddr{ \hfill} \\
%       \affaddr{ \hfill}\\
%       \email{ \hfill}
%}

\maketitle

\begin{abstract}
We consider optimizing average queueing delay and average power consumption in a nonpreemptive multi-class $M/G/1$ queue with dynamic power control that affects  instantaneous service rates. Four problems are studied: (1) satisfying per-class average  delay constraints; (2) minimizing a separable convex function of average delays subject to per-class delay constraints; (3) minimizing average power consumption subject to per-class delay constraints; (4) minimizing a separable convex function of average delays subject to an average power constraint. Combining an achievable region approach in queueing systems and the Lyapunov optimization theory suitable for optimizing dynamic systems with time average constraints, we propose a unified framework to solve the above problems. The solutions are variants of dynamic $c\mu$ rules, and implement weighted priority policies in every busy period, where weights are determined by past queueing delays in all job classes. Our solutions  require limited statistical knowledge of arrivals and service times, and no statistical knowledge is needed in the first problem. Overall, we provide a new set of tools for stochastic optimization and control over multi-class queueing systems with time average constraints.
\end{abstract}
\section{Introduction}
Stochastic scheduling over multi-class queueing systems  has important applications such as CPU scheduling, request processing in web servers, and QoS provisioning to different types of traffic in a telecommunication network. In these systems, power management is increasingly important due to their massive energy consumption. To study this problem, in this paper we consider a single-server multi-class queueing system whose instantaneous service rate is controllable by dynamic power allocations. This is modeled as a nonpreemptive multi-class $M/G/1$ queue with $N$ job classes $\{1, \ldots, N\}$, and the goal is to optimize average queueing delays of all job classes and average power consumption in this queueing network. We consider  four delay and power control problems:

\begin{enumerate}
\item  Designing a policy that yields average queueing delay $\oW_{n}$ of class $n$ satisfying $\oW_{n} \leq d_{n}$ for all classes, where $\{d_{1}, \ldots, d_{N}\}$ are given feasible delay bounds. Here we assume a fixed power allocation and no power control.
\item Minimizing a separable convex function $\sum_{n=1}^{N} f_{n}(\oW_{n})$ of average queueing delays $(\oW_{n})_{n=1}^{N}$ subject to delay constraints $\oW_{n} \leq d_{n}$ for all  classes $n$; assuming a fixed power allocation and no power control.
\item Under dynamic power allocation,  minimizing average power consumption subject to delay constraints $\oW_{n} \leq d_{n}$ for all  classes $n$.
\item Under dynamic power allocation, minimizing a separable convex function $\sum_{n=1}^{N} f_{n}(\oW_{n})$ of average queueing delays $(\oW_{n})_{n=1}^{N}$ subject to an average power constraint.
\end{enumerate}
These problems are presented with increasing complexity for the readers to gradually familiarize themselves with the methodology we use to attack these problems.

Each of the above problems is highly nontrivial, thus novel yet simple approaches are needed. This paper provides such a framework by connecting two powerful stochastic optimization theories: The \emph{achievable region approach} in queueing systems, and the \emph{Lyapunov optimization theory} in wireless networks. In queueing systems, the achievable region approach that treats optimal control problems as mathematical programming ones has been fruitful; see~\cite{Yao02lecture,  BPT94, Ber95, Nin09} for a detailed survey. In a nonpreemptive multi-class $M/G/1$ queue, it is known that the collection of all feasible average queueing delay vectors form a special polytope (a base of a polymatroid) with vertices being the performance vectors of strict priority policies (\cite{FaG88a}, see Section~\ref{sec:1104} for more details). As a result, every feasible average queueing delay vector is attainable by a randomization of strict priority policies. Such randomization can be implemented in a framed-based style, where a priority ordering is randomly deployed in every busy period using a probability distribution that is used in all busy periods (see Lemma~\ref{lem:2201} in Section~\ref{sec:1104}). This view of the  delay performance region is useful in the first two delay control problems.

In addition to queueing delay, when dynamic power control is part of the decision space, it is natural to consider dynamic policies that allocate a fixed power in every busy period. The resulting joint power and delay performance region is then spanned by frame-based randomizations of power control and strict priority policies. We treat the last two delay and power control problems as stochastic optimization  over such a performance region (see Section~\ref{sec:1103} for an example).

With the above characterization of performance regions, we solve the  four control problems using  \emph{Lyapunov optimization theory}. This theory is originally developed for stochastic optimal control over time-slotted wireless networks~\cite{TaE92,TaE93}, later extended by~\cite{GNT06,Nee03thesis} that allow optimizing various performance objectives such as average power~\cite{Nee06} or throughput utility~\cite{NML08},  and recently generalized to optimize dynamic systems that have a renewal structure~\cite{Nee09conf,Nee10conf,Nee10book,LaN10arXivb}.  The Lyapunov optimization theory transforms  time average constraints into virtual queues that need to be stabilized. Using a Lyapunov drift argument, we construct frame-based policies to solve the four control problems. The resulting policy is a sequence of \emph{base policies} implemented frame by frame, where the collection of all base policies span the performance region through time sharing or randomization.  The base policy used in each frame is chosen by minimizing a ratio of an expected ``drift plus penalty''  sum over the expected frame size, where the ratio is a function of past queueing delays in all job classes.  In this paper the base policies are nonpreemptive strict priority policies with deterministic power allocations. 

Our methodology is as follows. By characterizing the performance region using the collection of all randomizations of base policies, for each control problem, there exists an optimal random mixture of base policies that solves the problem. Although the probability distribution that defines the optimal random mixture is unknown, we construct a dynamic policy using Lyapunov optimization theory. This policy makes greedy decisions in every frame,  stabilizes all virtual queues (thus satisfying all time average constraints), and yields near-optimal performance. The \emph{existence} of the optimal randomized policy is essential to prove these results. 

%Our methodology is different from the tradition, which characterizes the performance region using linear and nonlinear constraints, and designs control algorithms using relaxations and/or primal-dual methods (see~\cite{Ber95} for an example).

In our policies for the four control problems, requests of different classes are prioritized by a dynamic $c\mu$ rule~\cite{Yao02lecture} which, in every busy period, assigns priorities in the decreasing order of weights associated with each class. The weights of all classes are updated at the end of every busy period by simple queue-like rules (so that different priorities may be assigned in different busy periods), which capture the running difference between the current and the  desired performance. The dynamic $c\mu$ rule in the first problem does not require any statistical knowledge of arrivals and service times. The policy for the second problem requires only the mean but not higher moments of arrivals and service times. In the last two problems with dynamic power control, beside the dynamic $c\mu$ rules, a power level is allocated  in every busy period by optimizing a weighted sum of power and power-dependent average delays. The policies for the third and the last problem require the mean and the first two moments of arrivals and service times, respectively, because of dynamic power allocations.

In each of the last three problems, our policies yield performance that is at most  $O(1/V)$ away from the optimal, where $V>0$ is a control parameter that can be chosen sufficiently large to yield near-optimal performance. The tradeoff of choosing large $V$ values is the amount of time  required to meet the time average constraints. In this paper we also propose a \emph{proportional delay fairness} criterion, in the same spirit as the well-known rate proportional fairness~\cite{Kel97} or utility proportional fairness~\cite{WPL06}, and show that the corresponding delay objective functions are quadratic. Overall, since our policies use dynamic $c\mu$ rules with weights of simple updates, and require limited statistical knowledge, they scale gracefully with the number of job classes and are suitable for online implementation.

In the literature, work~\cite{FaG88b} characterizes multi-class $G/M/c$ queues that have polymatroidal performance regions, and provides two numerical methods to minimize a separable convex function of average delays as an unconstrained static optimization problem. But in~\cite{FaG88b} it is unclear how to control the queueing system to achieve the  optimal performance. Minimizing a convex holding cost in a single-server multi-class queue is  formulated as a restless bandit problem in~\cite{AGN03,GLA03}, and Whittle's index policies~\cite{Whi88} are constructed as a heuristic solution. Work~\cite{Mie95}  proposes a generalized $c\mu$ rule to maximize a convex holding cost over a finite horizon in a multi-class queue, and shows it is asymptotically optimal under heavy traffic. This paper provides a dynamic control algorithm for the minimization of convex functions of average delays. Especially, we consider additional time average power and delay constraints, and our solutions require limited statistical knowledge and have provable near-optimal performance.

This paper also applies to power-aware scheduling problems in computer systems. These problems are widely studied in different contexts, where two main analytical tools are competitive analysis~\cite{ALW10conf, BCP09conf, AWT09conf,AaF07,BPS07conf} and $M/G/1$-type queueing theory (see~\cite{WAT09conf} and references therein), both used to optimize metrics such as a weighted sum of average power and delay. This paper presents a fundamentally different approach for more directed control over average power and delays, and considers a multi-class setup with time average constraints.

In the rest of the paper, the detailed queueing model is given in Section~\ref{sec:702}, followed by a summary of useful $M/G/1$ properties in Section~\ref{sec:1104}. The four delay-power control problems are solved in Section~\ref{sec:1803}-\ref{sec:1801}, followed by simulation results.

%For example, in~\cite{LaN10arXivb}, we consider optimizing a general functional objective over an inner bound on the performance region of a restless bandit problem with Markov ON/OFF bandits. The performance inner bound is spanned by randomizing round robin policies, and we develop greedy policies that executes a round robin policy for one round in each frame, where the frame size is the random duration of a round; the greedy policy yields near-optimal performance in the inner bound.

\section{Queueing Model} \label{sec:702}

We only consider \emph{queueing delay}, not \emph{system delay} (queueing plus service) in this paper. System delay can be easily incorporated since, in a nonpreemptive system,  average queueing and system delay differ only by a constant (the average service time). We will use ``delay'' and ``queueing delay'' interchangeably in the rest of the paper.
 
Consider a single-server queueing system processing jobs categorized into $N$ classes. In each class $n\in \{1, 2, \ldots, N\}$, jobs arrive as a Poisson process with rate $\lambda_{n}$. Each class $n$ job has size $S_{n}$. We assume $S_{n}$ is i.i.d. in each class, independent across classes, and that the first four moments of $S_{n}$ are  finite for all classes $n$. The system processes arrivals nonpreemptively with instantaneous service rate $\mu(P(t))$, where $\mu(\cdot)$ is a concave, continuous, and nondecreasing function of the allocated power $P(t)$ (the concavity of rate-power relationship is observed in computer systems~\cite{KaM08book,WAT09conf,GHD09conf}).  Within each class, arrivals are served in a first-in-first-out fashion. We consider a frame-based system, where each frame consists of an idle period and the following busy period. Let $t_{k}$ be the start of the $k$th frame for each $k\in\Z^{+}$; the $k$th frame is $[t_{k}, t_{k+1})$. Define $t_{0} = 0$ and assume the system is initially empty. Define $T_{k} \triangleq t_{k+1} - t_{k}$ as the size of frame $k$. Let $A_{n,k}$ denote the set of class $n$ arrivals in frame $k$. For each job $i\in A_{n,k}$, let $W_{n,k}^{(i)}$ denote its queueing delay.

The control over this queueing system is power allocations and job scheduling across all classes. We restrict to the following frame-based policies that are both \emph{causal} and \emph{work-conserving}:\footnote{\emph{Causality} means that every control decision depends only on the current and past states of the system; \emph{work-conserving} means that the server is never idle when there is still work to do.}
\begin{quote}
In every frame $k\in\Z^{+}$, use a fixed power level $P_{k} \in[\pmin, \pmax]$ and a  nonpreemptive strict priority policy $\pi(k)$ for the duration of the busy period in that frame. The decisions are possibly random.
\end{quote}
In these policies, $\pmax$ denotes the maximum power allocation. We assume $\pmax$ is finite, but sufficiently large to ensure feasibility of the desired delay constraints. The minimum power $\pmin$ is chosen to be large enough so that the queue is stable even if power $\pmin$ is used for all time. In particular, for stability we need 
\[
\sum_{n=1}^{N} \lambda_{n} \frac{\expect{S_{n}}}{ \mu(\pmin)} < 1 \Rightarrow \mu(\pmin) >  \sum_{n=1}^{N} \lambda_{n} \expect{S_{n}}.
\]
The strict priority rule $\pi(k) = (\pi_{n}(k))_{n=1}^{N}$ is represented by a permutation of $\{1, \ldots, N\}$, where  class $\pi_{n}(k)$ gets the $n$th highest priority.

The motivation of focusing on the above frame-based policies is to simplify the control of the queueing system to achieve complex performance objectives. Simulations in~\cite{AGM99}, however, suggest that this method may incur higher variance in performance than  policies that take control actions based on  job occupancies in the queue.  Yet, job-level scheduling seems difficult to attack problems considered in this paper. It may involve solving high-dimensional (partially observable) Markov decision processes with time average power and delay constraints and convex holding costs.

 \subsection{Definition of Average Delay} \label{sec:2404}
The average delay under policies we propose later may not have well-defined limits. Thus, inspired by~\cite{Nee10conf}, we define
\begin{equation} \label{eq:1103}
\oW_{n} \triangleq \limsup_{K\to\infty} \frac{  \expect{\sum_{k=0}^{K-1} \sum_{i\in A_{n,k}} W_{n,k}^{(i)}} }{\expect{\sum_{k=0}^{K-1} \abs{A_{n,k}}}}
\end{equation}
as the average delay of class $n\in\{1, \ldots, N\}$, where $\abs{A_{n,k}}$ is the number of class $n$ arrivals during frame $k$. We only consider delay sampled at frame boundaries for simplicity. To verify~\eqref{eq:1103}, note that the running average delay of class $n$ jobs up to time $t_{K}$ is equal to
\[
\frac{\sum_{k=0}^{K-1} \sum_{i\in A_{n,k}} W_{n,k}^{(i)}}{\sum_{k=0}^{K-1} \abs{A_{n,k}}} = \frac{\frac{1}{K}\sum_{k=0}^{K-1} \sum_{i\in A_{n,k}} W_{n,k}^{(i)}}{ \frac{1}{K} \sum_{k=0}^{K-1} \abs{A_{n,k}}}.
\]
Define
\[
w_{n}^{\text{av}} \triangleq \lim_{K\to\infty} \frac{1}{K} \sum_{k=0}^{K-1} \sum_{i\in A_{n,k}} W_{n,k}^{(i)}, \ a_{n}^{\text{av}} \triangleq \lim_{K\to\infty} \frac{1}{K} \sum_{k=0}^{K-1} \abs{A_{n,k}}.  \\
\]
If both limits $w_{n}^{\text{av}}$ and $a_{n}^{\text{av}}$ exist, the ratio $w_{n}^{\text{av}}/a_{n}^{\text{av}}$ is the limiting average delay for class $n$. In this case, we get
\begin{equation} \label{eq:1102}
\begin{split}
\oW_{n} &= \frac{ \lim_{K\to\infty} \expect{ \frac{1}{K} \sum_{k=0}^{K-1} \sum_{i\in A_{n,k}} W_{n,k}^{(i)}} }{\lim_{K\to\infty} \expect{ \frac{1}{K}\sum_{k=0}^{K-1} \abs{A_{n,k}}}} \\
&= \frac{  \expect{ \lim_{K\to\infty} \frac{1}{K} \sum_{k=0}^{K-1} \sum_{i\in A_{n,k}} W_{n,k}^{(i)}} }{ \expect{ \lim_{K\to\infty} \frac{1}{K}\sum_{k=0}^{K-1} \abs{A_{n,k}}}} = \frac{w_{n}^{\text{av}}}{a_{n}^{\text{av}}},
\end{split}
\end{equation}
which shows $\oW_{n}$ is indeed the limiting average delay.\footnote{The second equality in~\eqref{eq:1102}, where we pass the limit into the expectation, can be proved by a generalized Lebesgue's dominated convergence theorem stated as follows. Let $\{X_{n}\}_{n=1}^{\infty}$ and $\{Y_{n}\}_{n=1}^{\infty}$ be two sequences of random variables such that: (1) $ 0\leq \abs{X_{n}} \leq Y_{n}$ with probability $1$ for all $n$; (2) For some random variables $X$ and $Y$, $X_{n}\to X$ and $Y_{n}\to Y$ with probability $1$; (3) $\lim_{n\to\infty} \expect{Y_{n}} = \expect{Y} < \infty$. Then $\expect{X}$ is finite and $\lim_{n\to\infty} \expect{X_{n}} = \expect{X}$. The details are omitted for brevity.} The definition in~\eqref{eq:1103} replaces $\lim$ by $\limsup$ to guarantee it is well-defined.

%Let $\wt{A}_{n}^{(K)} \triangleq \cup_{k=1}^{K} A_{n,k}$ denote the set of class $n$ arrivals in the first $K$ frames.  We denote by $K(t)$ the number of full frames before time $t$, and by $N_{n}(t)$ the number of class $n$ jobs served up to time $t$. 

\section{Preliminaries} \label{sec:1104}

This section summarizes useful properties of a nonpreemptive multi-class $M/G/1$ queue. Here we assume a fixed power allocation $P$ and a fixed service rate $\mu(P)$ (this is extended in Section~\ref{sec:powercontrol}). Let $X_{n} \triangleq S_{n}/\mu(P)$ be the service time of a class $n$ job. Define $\rho_{n} \triangleq \lambda_{n} \expect{X_{n}}$. Fix an arrival rate vector $(\lambda_{n})_{n=1}^{N}$ satisfying $\sum_{n=1}^{N} \rho_{n} < 1$; the rate vector $(\lambda_{n})_{n=1}^{N}$ is supportable in the queueing network.

For each $k\in \Z^{+}$, let $I_{k}$ and $B_{k}$ denote the $k$th idle and busy period, respectively; the frame size $T_{k} = I_{k} + B_{k}$. The distribution of $B_{k}$ (and $T_{k}$) is fixed under any work-conserving policy, since the sample path of unfinished work in the system is independent of scheduling policies. Due to the memoryless property of Poisson arrivals, we have $\expect{I_{k}} = 1/(\sum_{n=1}^{N} \lambda_{n})$ for all $k$. For the same reason, the system renews itself at the start of each frame. Consequently, the frame size $T_{k}$, busy period $B_{k}$, and the per-frame job arrivals $\abs{A_{n,k}}$ of class $n$, are all i.i.d. over $k$. Using renewal reward theory~\cite{Ros96book} with renewal epochs defined at frame boundaries $\{t_{k}\}_{k=0}^{\infty}$, we have:
\begin{align}
&\expect{T_{k}} = \frac{\expect{I_{k}}}{1-\sum_{n=1}^{N} \rho_{n}} = \frac{1}{(1-\sum_{n=1}^{N} \rho_{n})\sum_{n=1}^{N} \lambda_{n}} \label{eq:2102} \\
&\expect{\abs{A_{n,k}}} = \lambda_{n} \expect{T_{k}},\, \forall n\in\{1, \ldots, N\},\, \forall k\in\Z^{+}. \label{eq:2101}
\end{align}

It is useful to consider the randomized policy $\pirand$ that is defined by a given probability distribution over all possible $N!$ priority orderings. Specifically, policy $\pirand$ randomly selects priorities at the beginning of every new frame according to this distribution, and implements the corresponding nonpreemptive priority rule for the duration of the frame. Again by renewal reward theory, the average queueing delays $(\oW_{n})_{n=1}^{N}$ rendered by a $\pirand$ policy satisfy in each frame $k\in\Z^{+}$:
\begin{equation} \label{eq:623}
\expect{\sum_{i\in A_{n,k}} W_{n,k}^{(i)}} = \expect{\int_{t_{k}}^{t_{k+1}} Q_{n}(t)\, \dd t} = \lambda_{n} \oW_{n} \expect{T_{k}},
\end{equation}
where we recall that $W_{n,k}^{(i)}$ represents only the queueing delay (not including service time), and $Q_{n}(t)$ denotes the number of class $n$ jobs waiting in the queue (not including that in the server) at time $t$.

Next we summarize useful properties of the performance region of average queueing delay vectors $(\oW_{n})_{n=1}^{N}$ in a nonpreemptive multi-class $M/G/1$ queue. For these results we refer readers to~\cite{FaG88a, Yao02lecture, BaG92book} for a detailed introduction. Define the value   $x_{n} \triangleq \rho_{n} \oW_{n}$ for each class $n\in\{1, \ldots, N\}$, and denote by $\Omega$ the performance region of the vector $(x_{n})_{n=1}^{N}$. The set $\Omega$ is a special polytope called \emph{(a base of) a polymatroid}~\cite{Wel76book}. An important property of the polymatroid $\Omega$ is: (1) Each vertex of $\Omega$ is the performance vector of a strict nonpreemptive priority rule; (2) Conversely, the performance vector of each strict nonpreemptive priority rule is a vertex of $\Omega$. In other words, there is a one-to-one mapping between vertices of $\Omega$ and the set of strict nonpreemptive priority rules. As a result, every feasible performance vector $(x_{n})_{n=1}^{N} \in \Omega$, or equivalently every feasible queueing delay vector $(\oW_{n})_{n=1}^{N}$, is attained by a randomization of strict nonpreemptive priority policies. For completeness, we formalize the last known result in the next lemma.

\begin{lem} \label{lem:2201}
In a nonpreemptive multi-class $M/G/1$ queue, define
\[
\mathcal{W} \triangleq \left\{ (\oW_{n})_{n=1}^{N} \mid (\rho_{n} \oW_{n})_{n=1}^{N} \in \Omega \right\}
\]
as the performance region~\cite{FaG88a} of average queueing delays. Then:
\begin{enumerate}
\item The performance vector $(\oW_{n})_{n=1}^{N}$ of each frame-based randomized policy $\pirand$ is in the delay region $\mathcal{W}$.
\item Conversely, every vector $(\oW_{n})_{n=1}^{N}$ in the delay region $\mathcal{W}$ is the performance vector of a $\pirand$ policy.
\end{enumerate}
\end{lem}

\begin{proof}[Proof of Lemma~\ref{lem:2201}]
 Given in Appendix~\ref{sec:2204}.
 \end{proof}
 
Optimizing a linear function over the polymatroidal region $\Omega$ will be useful. The solution is the following $c\mu$ rule:

 \begin{lem}[The $c\mu$ rule~\cite{Yao02lecture,BaG92book}] \label{lem:603}
In a nonpreemptive multi-class $M/G/1$ queue, define $x_{n} \triangleq \rho_{n} \oW_{n}$ and consider the linear program:
\begin{align}
\text{minimize:} &\quad  \sum_{n=1}^{N} c_{n}\, x_{n}  \label{eq:1301} \\
\text{subject to:} &\quad (x_{n})_{n=1}^{N} \in \Omega \label{eq:1302}
\end{align}
where $c_{n}$ are nonnegative constants. We assume $\sum_{n=1}^{N} \rho_{n} < 1$ for stability, and that second moments $\expect{X_{n}^{2}}$ of service times are finite for all classes $n$. The optimal solution to~\eqref{eq:1301}-\eqref{eq:1302} is a strict nonpreemptive priority policy that assigns priorities  in the decreasing order of $c_{n}$. That says, if  $c_{1} \geq c_{2} \geq \ldots \geq c_{N}$, then class $1$ gets the highest priority, class $2$ gets the second highest priority, and so on. In this case, the optimal average queueing delay $\oW_{n}^{*}$ of class $n$ is
\[
\oW_{n}^{*} = \frac{R}{(1-\sum_{k=0}^{n-1} \rho_{k})(1-\sum_{k=0}^{n} \rho_{k})},
\]
where $\rho_{0} \triangleq 0$ and $R\triangleq \frac{1}{2} \sum_{n=1}^{N} \lambda_{n} \expect{X_{n}^{2}}$.
\end{lem}

\section{Achieving Delay Constraints} \label{sec:1803}

The first problem we consider is to construct a frame-based policy that yields average delays satisfying $\oW_{n} \leq d_{n}$ for all classes $n \in \{1, \ldots, N\}$, where $d_{n}>0$ are given constants. We assume a fixed power allocation and that the delay constraints are feasible.

Our solution relies on tracking the running difference between past queueing delays for each class $n$ and the desired delay bound $d_{n}$. For each class $n\in\{1, \ldots, N\}$, we define a discrete-time \emph{virtual delay queue} $\{Z_{n,k}\}_{k=0}^{\infty}$ where $Z_{n,k+1}$ is updated at frame boundary $t_{k+1}$ following the equation
\begin{equation} \label{eq:601}
Z_{n,k+1} = \max\left[ Z_{n,k} + \sum_{i\in A_{n,k}} \left(W_{n,k}^{(i)} - d_{n}\right), \, 0\right].
\end{equation}
Assume $Z_{n,0} = 0$ for all $n$. In~\eqref{eq:601}, the delays $W_{n,k}^{(i)}$ and constant $d_{n}$ can viewed as arrivals and service of the queue $\{Z_{n,k}\}_{k=0}^{\infty}$, respectively. If this queue is stabilized, we know that the average arrival rate to the queue (being the per-frame average sum of class $n$ delays $\sum_{n\in A_{n,k}} W_{n,k}^{(i)}$) is less than or equal to the average service rate (being the value $d_{n}$ multiplied by the average number of class $n$ arrivals per frame), from which we infer $\oW_{n} \leq d_{n}$. This is formalized below.

\begin{defn}
We say queue $\{Z_{n,k}\}_{k=0}^{\infty}$ is mean rate stable if  $\lim_{K\to\infty} \expect{Z_{n,K}}/ K = 0$.
\end{defn}
\begin{lem} \label{lem:602}
If queue $\{Z_{n,k}\}_{k=0}^{\infty}$ is mean rate stable, then $\oW_{n} \leq d_{n}$.
\end{lem}

\begin{proof}[Proof of Lemma~\ref{lem:602}]
From~\eqref{eq:601} we get
\[
Z_{n,k+1} \geq Z_{n,k} - d_{n} \abs{A_{n,k}} + \sum_{i\in A_{n,k}} W_{n,k}^{(i)}.
\]
Summing the above over $k\in\{0, \ldots, K-1\}$, using $Z_{n,0}=0$, and taking expectation yields
\[
\expect{Z_{n,K}} \geq - d_{n} \expect{\sum_{k=0}^{K-1} \abs{A_{n,k}}} + \expect{\sum_{k=0}^{K-1} \sum_{i\in A_{n,k}} W_{n,k}^{(i)}}.
\]
Dividing the above by $\expect{\sum_{k=0}^{K-1} \abs{A_{n,k}}}$ yields
\[
\frac{\expect{Z_{n,K}}}{\expect{\sum_{k=0}^{K-1} \abs{A_{n,k}}}} \geq \frac{ \expect{\sum_{k=0}^{K-1} \sum_{i\in A_{n,k}} W_{n,k}^{(i)}}}{\expect{\sum_{k=0}^{K-1} \abs{A_{n,k}}}} - d_{n}.
\]
Taking a $\limsup$ as $K\to\infty$ and using~\eqref{eq:1103} yields
\[
\begin{split}
\oW_{n} \leq d_{n} + \limsup_{K\to\infty} \frac{\expect{Z_{n,K}}}{K} \frac{K}{\expect{\sum_{k=0}^{K-1} \abs{A_{n,k}}}}.
\end{split}
\]
Using $\expect{\abs{A_{n,k}}} = \lambda_{n} \expect{T_{k}} \geq \lambda_{n} \expect{I_{k}} = \lambda_{n} \expect{I_{0}}$, we get
\[
\oW_{n} \leq d_{n} + \frac{1}{\lambda_{n} \expect{I_{0}}} \lim_{K\to\infty} \frac{\expect{Z_{n,K}}}{K} = d_{n}
\]
by mean rate stability of $Z_{n,k}$.
\end{proof}

\subsection{Delay Feasible Policy} \label{sec:701}
The following policy stabilizes every $\{Z_{n,k}\}_{k=0}^{\infty}$ queue in the mean rate stable sense and thus achieves $\oW_{n}\leq d_{n}$ for all classes $n$.

%\subsubsection*{Delay Feasible ($\mathsf{DelayFeas}$) Policy:}
\underline{\textit {Delay Feasible ($\mathsf{DelayFeas}$) Policy:}}

\begin{itemize}
\item In every frame $k\in\Z^{+}$, update $Z_{n,k}$ by~\eqref{eq:601} and serve jobs using nonpreemptive strict priorities assigned in the decreasing order of $Z_{n,k}$; ties are broken arbitrarily.
\end{itemize}
We note that the $\mathsf{DelayFeas}$ policy does not require any statistical knowledge of job arrivals and service times. Intuitively, each $Z_{n,k}$ queue tracks the amount of past queueing delays in class $n$ exceeding the desired delay bound $d_{n}$ (see~\eqref{eq:601}), and the $\mathsf{DelayFeas}$ policy gives priorities to classes that more severely violate their delay constraints.

\subsection{Motivation of the $\mathsf{DelayFeas}$ Policy} \label{sec:2202}
The structure of the $\mathsf{DelayFeas}$ policy follows a Lyapunov drift argument. Define vector $\bZ_{k} \triangleq (Z_{n,k})_{n=1}^{N}$. For some finite constants $\theta_{n} >0$ for all classes $n$,  we define the \emph{quadratic Lyapunov function}
\[
L(\bZ_{k}) \triangleq \frac{1}{2} \sum_{n=1}^{N} \theta_{n} Z_{n,k}^{2}
\]
as a weighted scalar measure of queue sizes $(Z_{n,k})_{n=1}^{N}$. Define the \emph{one-frame Lyapunov drift}
\[
\Delta(\bZ_{k}) \triangleq \expect{L(\bZ_{k+1}) - L(\bZ_{k}) \mid \bZ_{k}}
\]
as the conditional expected difference of $L(\bZ_{k})$ over a frame. Taking square of~\eqref{eq:601} and using $(\max[a, 0])^{2} \leq a^{2}$ yields
\begin{equation} \label{eq:1602}
Z_{n,k+1}^{2} \leq \left[Z_{n,k} + \sum_{i\in A_{n,k}} \left(W_{n,k}^{(i)} - d_{n}\right)\right]^{2}.
\end{equation}
Multiplying~\eqref{eq:1602} by $\theta_{n}/2$, summing over $n\in\{1, \ldots, N\}$, and taking conditional expectation on $\bZ_{k}$, we get
\begin{equation} \label{eq:104}
\begin{split}
\Delta(\bZ_{k})
			&\leq \frac{1}{2} \sum_{n=1}^{N} \theta_{n}\, \expect{ \left( \sum_{i\in A_{n,k}} \left(W_{n,k}^{(i)} - d_{n}\right) \right)^{2} \mid \bZ_{k}}  \\
			&\quad + \sum_{n=1}^{N} \theta_{n} \, Z_{n,k}\, \expect{\sum_{i\in A_{n,k}} \paren{W_{n,k}^{(i)}-d_{n}} \mid \bZ_{k}}.
\end{split}
\end{equation}
Lemma~\ref{lem:608} in Appendix~\ref{appendix:701} shows that the second term of~\eqref{eq:104} is bounded by a finite constant $C>0$. It leads to the following Lyapunov drift inequality:
\begin{equation} \label{eq:303}
\Delta(\bZ_{k}) \leq C + \sum_{n=1}^{N} \theta_{n}\, Z_{n,k} \, \expect{\sum_{i\in A_{n,k}} \paren{W_{n,k}^{(i)}-d_{n}} \mid \bZ_{k}}.
\end{equation}
%\begin{equation} \label{eq:303}
%\begin{split}
%\Delta(\bZ_{k}) &\leq C + \sum_{n=1}^{N} Z_{n,k} \expect{\sum_{i\in A_{n,k}} \paren{W_{n,k}^{(i)}-d_{n}} \mid \bZ_{k}} \\
%&= C - \expect{T_{k}} \sum_{n=1}^{N} Z_{n,k}\, \lambda_{n}\, d_{n} \\
%&\quad + \sum_{n=1}^{N} Z_{n,k} \, \expect{ \sum_{i\in A_{n,k}} W_{n,k}^{(i)} \mid \bZ_{k}}
%\end{split}
%\end{equation}
%where the equality uses~\eqref{eq:1002}.

Over all frame-based policies, we are interested in the one that, in each frame $k$ after observing $\bZ_{k}$, minimizes the right side of~\eqref{eq:303}.  Recall that our policy on frame $k$ chooses which nonpreemptive priorities to use during the frame. To show that this is exactly the $\mathsf{DelayFeas}$ policy, we simplify~\eqref{eq:303}. Under a frame-based policy, we have by renewal reward theory
\[
\expect{ \sum_{i\in A_{n,k}} W_{n,k}^{(i)} \mid \bZ_{k}} = \lambda_{n} \oW_{n,k}\, \expect{T_{k}},
\]
where $\oW_{n,k}$ denotes the long-term average delay of class $n$ if the control in frame $k$ is repeated in every frame. Together with $\expect{\abs{A_{n,k}}} = \lambda_{n} \expect{T_{k}}$, inequality~\eqref{eq:303} is re-written as
\begin{equation} \label{eq:1801}
\begin{split}
\Delta(\bZ_{k}) &\leq \left(C - \expect{T_{k}} \sum_{n=1}^{N} \theta_{n} \, Z_{n,k} \lambda_{n} d_{n} \right) \\
&\quad + \expect{T_{k}} \sum_{n=1}^{N} \theta_{n}\, Z_{n,k} \, \lambda_{n} \, \oW_{n,k}.
\end{split}
\end{equation}
Because in this section we do not have dynamic power allocation (so that power is fixed to the same value in every busy period), the value $\expect{T_{k}}$ is the same for all job scheduling policies. Then our desired policy, in every frame $k$, chooses a job scheduling to minimize the metric
$
\sum_{n=1}^{N} \theta_{n}\, Z_{n,k}\, \lambda_{n} \oW_{n,k}
$
over all feasible delay vectors $(\oW_{n,k})_{n=1}^{N}$. If we choose $\theta_{n} = \expect{X_{n}}$ for all classes $n$,\footnote{We note that the mean service time $\expect{X_{n}}$ as a value of $\theta_{n}$ is only needed in the arguments  constructing the $\mathsf{DelayFeas}$ policy. The $\mathsf{DelayFeas}$ policy itself does not need the knowledge of  $\expect{X_{n}}$.} the desired policy minimizes $\sum_{n=1}^{N} Z_{n,k}\, \lambda_{n}\, \expect{X_{n}} \oW_{n,k}$ in every frame $k$. From lemma~\ref{lem:603}, this  is achieved by the priority service rule defined by the $\mathsf{DelayFeas}$ policy.

\subsection{Performance of the $\mathsf{DelayFeas}$ Policy} \label{sec:1802}
\begin{thm} \label{thm:601}
For every collection of feasible delay bounds $\{d_{1}, \ldots, d_{N}\}$, the $\mathsf{DelayFeas}$ policy yields average delays satisfying $\oW_{n} \leq d_{n}$ for all classes $n\in\{1, \ldots, N\}$.
\end{thm}

\begin{proof}[Proof of Theorem~\ref{thm:601}]
It suffices to show that the $\mathsf{DelayFeas}$ policy yields mean rate stability for all $Z_{n,k}$ queues by Lemma~\ref{lem:602}. By Lemma~\ref{lem:2201}, there exists a randomized priority policy $\pirand^{*}$ (introduced in Section~\ref{sec:1104}) that yields average delays $\oW_{n}^{*}$ satisfying $\oW_{n}^{*} \leq d_{n}$ for all classes $n$. Since the $\mathsf{DelayFeas}$ policy minimizes the last term of~\eqref{eq:1801} in each frame (under $\theta_{n} = \expect{X_{n}}$ for all $n$),  comparing the $\mathsf{DelayFeas}$ policy with the $\pirand^{*}$ policy yields, in every frame $k$,
\[
\sum_{n=1}^{N} \theta_{n}\, Z_{n,k} \,\lambda_{n} \oW_{n,k}^{\mathsf{DelayFeas}}  \leq  \sum_{n=1}^{N} \theta_{n} \, Z_{n,k} \,\lambda_{n} \oW_{n}^{*}.
\]
It follows that~\eqref{eq:1801} under the $\mathsf{DelayFeas}$ policy is further upper bounded by
\[
\begin{split}
\Delta(\bZ_{k}) &\leq C +  \expect{T_{k}} \sum_{n=1}^{N} \theta_{n}\, Z_{n,k}\, \lambda_{n} (\oW_{n,k}^{\mathsf{DelayFeas}} - d_{n})  \\
&\leq C +  \expect{T_{k}} \sum_{n=1}^{N} \theta_{n}\, Z_{n,k}\, \lambda_{n} (\oW_{n}^{*} - d_{n}) \leq C.
\end{split}
\]
Taking expectation, summing over $k\in\{0, \ldots, K-1\}$, and noting $L(\bZ_{0}) = 0$, we get
\[
\expect{L(\bZ_{K})} = \frac{1}{2} \sum_{n=1}^{N} \theta_{n} \expect{Z_{n,K}^{2}} \leq KC.
\]
It follows that
$
\expect{Z_{n,K}^{2}} \leq 2KC/\theta_{n}
$
for all classes $n$. Since $Z_{n,K}\geq 0$, we get
\[
0\leq \expect{Z_{n,K}} \leq \sqrt{\expect{Z_{n,K}^{2}}} \leq \sqrt{2KC/\theta_{n}}.
\]
Dividing the above by $K$ and passing $K\to\infty$ yields
\[
\lim_{K\to\infty} \frac{\expect{Z_{n,K}}}{K} = 0, \quad \forall n\in\{1, \ldots, N\},
\]
and all $Z_{n,k}$ queues are mean rate stable.
\end{proof}

\section{Minimizing Delay Penalty Functions}  \label{sec:1804}

Generalizing the first delay feasibility problem, next we optimize a separable penalty function of average delays. For each class $n$, let $f_{n}(\cdot)$ be a nondecreasing, nonnegative, continuous, and convex function of average delay $\oW_{n}$. Consider the constrained penalty minimization problem
\begin{align}
\text{minimize:} 	&\quad \sum_{n=1}^{N} f_{n}(\oW_{n})  \label{eq:613} \\
\text{subject to:}	&\quad \oW_{n} \leq d_{n}, \quad \forall n\in\{1, \ldots, N\}. \label{eq:617}
\end{align}
We assume that a constant power is allocated in all frames, and that constraints~\eqref{eq:617} are feasible. The goal is to construct a frame-based policy that solves~\eqref{eq:613}-\eqref{eq:617}. Let $(\oW_{n}^{*})_{n=1}^{N}$ be the optimal solution to~\eqref{eq:613}-\eqref{eq:617}, attained by a randomized priority policy $\pirand^{*}$ (by Lemma~\ref{lem:2201}).

\subsection{Delay Proportional Fairness} \label{sec:703}

One interesting penalty function is the one that attains \emph{proportional fairness}. We say a delay vector $(\oW_{n}^{*})_{n=1}^{N}$ is \emph{delay proportional fair} if it is optimal under the quadratic penalty function $f_{n}(\oW_{n}) = \frac{1}{2} c_{n} \oW_{n}^{2}$ for each class $n$, where $c_{n}>0$ are given constants. The intuition is two-fold. First, under the quadratic penalty functions, any feasible delay vector $(\oW_{n})_{n=1}^{N}$ necessarily satisfies
\begin{equation} \label{eq:621}
\sum_{n=1}^{N} f_{n}'(\oW_{n}^{*}) (\oW_{n}-\oW_{n}^{*}) = \sum_{n=1} c_{n} (\oW_{n} - \oW_{n}^{*}) \oW_{n}^{*} \geq 0,
\end{equation}
which is analogous to the \emph{rate proportional fair}~\cite{Kel97} criterion
\begin{equation} \label{eq:620}
\sum_{n=1}^{N} c_{n} \frac{x_{n} - x_{n}^{*}}{x_{n}^{*}} \leq 0,
\end{equation}
where $(x_{n})_{n=1}^{N}$ is any feasible rate vector and $(x_{n}^{*})_{n=1}^{N}$ is the optimal rate vector. Second,  rate proportional fairness, when deviating from the optimal solution, yields the aggregate change of proportional rates less than or equal to zero (see~\eqref{eq:620}); it penalizes large rates to increase. When delay proportional fairness deviates from the optimal solution, the aggregate change of proportional delays is always nonnegative (see~\eqref{eq:621}); small delays are penalized for trying to improve.

\subsection{Delay Fairness Policy}
In addition to having the $\{Z_{n,k}\}_{k=0}^{\infty}$ queues updated by~\eqref{eq:601} for all classes $n$, we setup new discrete-time virtual queues $\{Y_{n,k}\}_{k=0}^{\infty}$ for all classes $n$, where $Y_{n,k+1}$ is updated at frame boundary $t_{k+1}$ as:
\begin{equation} \label{eq:622}
Y_{n,k+1} = \max\left[ Y_{n,k} + \sum_{i\in A_{n,k}} \left(W_{n,k}^{(i)} - r_{n,k} \right), 0\right],
\end{equation}
where $r_{n,k} \in[0, d_{n}]$ are auxiliary variables chosen at time $t_{k}$ independent of frame size $T_{k}$ and the number $\abs{A_{n,k}}$ of class $n$ arrivals in frame $k$. Assume $Y_{n,0} = 0$ for all $n$. Whereas the $Z_{n,k}$ queues are useful to enforce delay constraints $\oW_{n}\leq d_{n}$ (as seen in Section~\ref{sec:1803}), the $Y_{n,k}$ queues are useful to achieve the optimal delay vector $(\oW_{n}^{*})_{n=1}^{N}$.

%\subsubsection*{Delay Fairness ($\mathsf{DelayFair}$) Policy:}
\underline{\textit{Delay Fairness ($\mathsf{DelayFair}$) Policy:}}

\begin{enumerate}
\item In the $k$th frame for each $k\in\Z^{+}$, after observing $\bZ_{k}$ and $\bY_{k}$, use  nonpreemptive strict priorities assigned in the decreasing order of $(Z_{n,k} + Y_{n,k})/\expect{S_{n}}$, where $\expect{S_{n}}$ is the mean size of a class $n$ job. Ties are broken arbitrarily. 
\item At the end of the $k$th frame, compute $Z_{n,k+1}$ and $Y_{n,k+1}$ for all classes $n$ by~\eqref{eq:601} and \eqref{eq:622}, respectively, where $r_{n,k}$ is the solution to the convex program:
\begin{align*}
\text{minimize:}		&\quad V f_{n}(r_{n,k}) - Y_{n,k}\, \lambda_{n} \, r_{n,k} \\
\text{subject to:}	&\quad 0\leq r_{n,k}\leq d_{n},
\end{align*}
where $V>0$ is a predefined control parameter.
\end{enumerate}
While the $\mathsf{DelayFeas}$ policy in Section~\ref{sec:1803} does not require any statistical knowledge of arrivals and service times, the $\mathsf{DelayFair}$ policy needs the mean but not higher moments of arrivals and service times for all classes $n$.

In the example of delay proportional fairness with quadratic penalty functions $f_{n}(\oW_{n}) = \frac{1}{2} c_{n} \oW_{n}^{2}$ for all classes $n$, the second step of the $\mathsf{DelayFair}$ policy solves:
\begin{align*}
\text{minimize:}		&\quad \left(\frac{1}{2} V \,c_{n} \right) r_{n,k}^{2} - Y_{n,k} \, \lambda_{n} \, r_{n,k}\\
\text{subject to:}	&\quad 0\leq r_{n,k} \leq d_{n}.
\end{align*}
The solution is $r_{n,k}^{*} = \min \left[d_{n}, \frac{Y_{n,k} \lambda_{n}}{V\, c_{n}} \right]$.

\subsection{Motivation of the $\mathsf{DelayFair}$ Policy}

The $\mathsf{DelayFair}$ policy follows a Lyapunov drift argument similar to that in Section~\ref{sec:1803}. Define $\bZ_{k} \triangleq (Z_{n,k})_{n=1}^{N}$ and $\bY_{k} \triangleq (Y_{n,k})_{n=1}^{N}$. Define the Lyapunov function $L(\bZ_{k}, \bY_{k}) \triangleq \frac{1}{2} \sum_{n=1}^{N} (Z_{n,k}^{2} + Y_{n,k}^{2})$ and the one-frame Lyapunov drift
\[
\Delta(\bZ_{k}, \bY_{k}) \triangleq \expect{L(\bZ_{k+1}, \bY_{k+1}) - L(\bZ_{k}, \bY_{k})\mid \bZ_{k}, \bY_{k}}.
\]
Taking square of~\eqref{eq:622} yields
\begin{equation} \label{eq:1603}
Y_{n,k+1}^{2} \leq \left[ Y_{n,k} + \sum_{i\in A_{n,k}} \left(W_{n,k}^{(i)} - r_{n,k} \right) \right]^{2}.
\end{equation}
Summing~\eqref{eq:1602} and~\eqref{eq:1603} over all classes $n\in\{1, \ldots, N\}$, dividing the result by $2$, and taking conditional expectation on $\bZ_{k}$ and $\bY_{k}$, we get
\begin{equation} \label{eq:616}
\begin{split}
&\Delta(\bZ_{k}, \bY_{k}) 
		\leq C - \sum_{n=1}^{N} Z_{n,k}\, d_{n}\, \expect{\abs{A_{n,k}} \mid \bZ_{k}, \bY_{k} } \\
	&\quad	- \sum_{n=1}^{N}  Y_{n,k}\, \expect{r_{n,k} \abs{A_{n,k}} \mid \bZ_{k}, \bY_{k} } \\
	&\quad	+ \sum_{n=1}^{N} (Z_{n,k}+Y_{n,k})\, \expect{  \sum_{i\in A_{n,k}} W_{n,k}^{(i)} \mid \bZ_{k}, \bY_{k} },
\end{split}
\end{equation}
where $C>0$ is a finite constant, different from that used in Section~\ref{sec:2202}, upper bounding the sum of all $(\bZ_{k}, \bY_{k})$-independent terms. This constant exists using arguments similar to those in Lemma~\ref{lem:608} of Appendix~\ref{appendix:701}.

Adding to both sides of~\eqref{eq:616} the weighted penalty term
$
V \sum_{n=1}^{N} \expect{ f_{n}(r_{n,k})\, T_{k}\mid \bZ_{k}, \bY_{k}}
$,
where $V>0$ is a predefined control parameter, and evaluating the result under a frame-based policy (similar as the analysis in Section~\ref{sec:1802}), we get the following Lyapunov \emph{drift plus penalty} inequality:
\begin{equation} \label{eq:614} 
\begin{split}
&		\Delta(\bZ_{k}, \bY_{k}) + V \sum_{n=1}^{N} \expect{ f_{n}(r_{n,k})\, T_{k}\mid \bZ_{k}, \bY_{k}} \\
& \leq \left( C - \expect{T_{k}} \sum_{n=1}^{N} Z_{n,k} \, \lambda_{n} \, d_{n} \right) \\
& + \expect{T_{k}} \sum_{n=1}^{N} \expect{ V \, f_{n}(r_{n,k}) -  Y_{n,k}\,  \lambda_{n} \, r_{n,k}\mid \bZ_{k}, \bY_{k}} \\
&+ \expect{T_{k}} \sum_{n=1}^{N} (Z_{n,k}+Y_{n,k}) \lambda_{n} \oW_{n,k}.
\end{split}
\end{equation}
We are interested in minimizing the right side of~\eqref{eq:614} in every frame $k$ over all frame-based policies and (possibly random) choices of $r_{n,k}$. Recall that in this section a constant power is allocated in all frames so that the value $\expect{T_{k}}$ is fixed under any work-conserving policy. The first and second step of the $\mathsf{DelayFair}$ policy minimizes the last (by Lemma~\ref{lem:603}) and the second-to-last  term of~\eqref{eq:614}, respectively.

\subsection{Performance of the $\mathsf{DelayFair}$ Policy} \label{sec:1101}

\begin{thm} \label{thm:602}
Given any feasible delay bounds $\{d_{1}, \ldots, d_{N}\}$,  the $\mathsf{DelayFair}$ policy yields average delays satisfying $\oW_{n} \leq d_{n}$ for all classes $n\in\{1, \ldots, N\}$, and attains average delay penalty satisfying
\[
\begin{split}
& \limsup_{K\to\infty}\sum_{n=1}^{N} f_{n}\left( \frac{\expect{\sum_{k=0}^{K-1} \sum_{i\in A_{n,k}} W_{n,k}^{(i)}}}{\expect{\sum_{k=0}^{K-1} \abs{A_{n,k}}}}\right)  \\
&\quad  \leq \frac{C \sum_{n=1}^{N} \lambda_{n}}{V} + \sum_{n=1}^{N} f_{n}(\oW_{n}^{*}),
\end{split}
\]
where $V>0$ is a predefined control parameter and $C>0$ a finite constant. By choosing $V$ sufficiently large, we attain arbitrarily close to the optimal delay penalty $ \sum_{n=1}^{N} f_{n}(\oW_{n}^{*})$.
\end{thm}

We remark that the tradeoff of choosing a large $V$ value is the amount of time required for   virtual queues $\{Z_{n,k}\}_{k=0}^{\infty}$ and $\{Y_{n,k}\}_{k=0}^{\infty}$ to approach mean rate stability (see~\eqref{eq:2002} in the next proof), that is, the time required for the virtual  queue backlogs to be negligible with respect to the time horizon.

\begin{proof}[Proof of Theorem~\ref{thm:602}]
Consider the optimal randomized policy $\pirand^{*}$ that yields optimal delays $\oW_{n}^{*} \leq d_{n}$ for all classes $n$. Since the $\mathsf{DelayFair}$ policy minimizes the right side of~\eqref{eq:614}, comparing the $\mathsf{DelayFair}$ policy with the policy $\pirand^{*}$ and with the genie decision $r_{n,k}^{*}= \oW_{n}^{*}$ for all classes $n$ and frames $k$, inequality~\eqref{eq:614} under the $\mathsf{DelayFair}$ policy is further upper bounded by
%\begin{equation} \label{eq:1005}
%\begin{split}
%&		\Delta(\bZ_{k}, \bY_{k}) + V \sum_{n=1}^{N} \expect{ f_{n}(r_{n,k})\, T_{k} \mid \bZ_{k}, \bY_{k}} \\
%& 	\leq C - \expect{T_{k}} \sum_{n=1}^{N} Z_{n,k}  \lambda_{n} d_{n} + \expect{T_{k}} \sum_{n=1}^{N}(Z_{n,k}+Y_{n,k}) \lambda_{n} \oW_{n}^{*} \\
%&\quad	+ \expect{T_{k}} \sum_{n=1}^{N} \left(V f_{n}(\oW_{n}^{*}) - Y_{n,k} \, \lambda_{n} \oW_{n}^{*} \right)\\
%&	\leq C + V  \expect{T_{k}} \sum_{n=1}^{N} f_{n}(\oW_{n}^{*}).
%\end{split}
%\end{equation}
\begin{align} 
&		\Delta(\bZ_{k}, \bY_{k}) + V \sum_{n=1}^{N} \expect{ f_{n}(r_{n,k})\, T_{k} \mid \bZ_{k}, \bY_{k}} \notag \\
& 	\leq C - \expect{T_{k}} \sum_{n=1}^{N} Z_{n,k}\,  \lambda_{n} \,d_{n} + \expect{T_{k}} \sum_{n=1}^{N}(Z_{n,k}+Y_{n,k}) \lambda_{n} \oW_{n}^{*} \notag \\
&\quad	+ \expect{T_{k}} \sum_{n=1}^{N} \left(V f_{n}(\oW_{n}^{*}) - Y_{n,k} \, \lambda_{n} \oW_{n}^{*} \right) \notag \\
&	\leq C + V  \expect{T_{k}} \sum_{n=1}^{N} f_{n}(\oW_{n}^{*}). \label{eq:1005}
\end{align}
Removing the second term of~\eqref{eq:1005} yields
\begin{equation} \label{eq:2001}
\Delta(\bZ_{k}, \bY_{k}) \leq C + V \expect{T_{k}} \sum_{n=1}^{N} f_{n}(\oW_{n}^{*}) \leq C+VD,
\end{equation}
where $D \triangleq \expect{T_{k}} \sum_{n=1}^{N} f_{n}(\oW_{n}^{*})$ is a finite constant. Taking expectation of~\eqref{eq:2001}, summing over $k\in\{0, \ldots, K-1\}$, and noting $L(\bZ_{0}, \bY_{0}) =0$ yields $\expect{L(\bZ_{K}, \bY_{K})} \leq K(C + VD)$. It follows that, for each class $n$ queue $\{Z_{n,k}\}_{k=0}^{\infty}$, we have
\begin{equation} \label{eq:2002}
\begin{split}
0 \leq \frac{\expect{Z_{n,K}}}{K} &\leq \sqrt{\frac{\expect{Z_{n,K}^{2}}}{K^{2}}} \\
&\leq \sqrt{\frac{2 \expect{L(\bZ_{k}, \bY_{K})}}{K^{2}}} \leq \sqrt{\frac{2C}{K} + \frac{2VD}{K}}.
\end{split}
\end{equation}
Passing $K\to\infty$ proves that queue $\{Z_{n,k}\}_{k=0}^{\infty}$ is mean rate stable for all classes $n$. Thus constraints $\oW_{n} \leq d_{n}$ are satisfied by Lemma~\ref{lem:602}. Similarly, the $\{Y_{n,k}\}_{k=0}^{\infty}$ queues are mean rate stable for all classes $n$.

Next, taking expectation of~\eqref{eq:1005}, summing over $k\in\{0, \ldots, K-1\}$, dividing by $V$, and noting $L(\bZ_{0}, \bY_{0})=0$ yields
\[
\begin{split}
&\frac{\expect{L(\bZ_{K}, \bY_{K})}}{V} + \sum_{n=1}^{N} \expect{\sum_{k=0}^{K-1}  f_{n}(r_{n,k})\, T_{k} } \\
&\qquad \leq \frac{KC}{V} + \expect{\sum_{k=0}^{K-1} T_{k}} \sum_{n=1}^{N} f_{n}(\oW_{n}^{*}). 
\end{split}
\]
Removing the first term and dividing by $\expect{\sum_{k=0}^{K-1} T_{k}}$ yields
\begin{align}
\sum_{n=1}^{N}
&\frac{ \expect{\sum_{k=0}^{K-1}   f_{n}(r_{n,k})\,T_{k}}}{\expect{\sum_{k=0}^{K-1} T_{k}}} \leq \frac{KC}{V\expect{\sum_{k=0}^{K-1} T_{k}}} + \sum_{n=1}^{N} f_{n}(\oW_{n}^{*}) \notag \\
&\stackrel{(a)}{\leq} \frac{C\sum_{n=1}^{N}\lambda_{n}}{V}  + \sum_{n=1}^{N} f_{n}(\oW_{n}^{*}), \label{eq:1105}
\end{align}
where (a) follows $\expect{T_{k}} \geq \expect{I_{k}} = 1/(\sum_{n=1}^{N} \lambda_{n})$. By~\cite[Lemma~$7.6$]{Nee10book} and convexity of $f_{n}(\cdot)$, we get
\begin{equation} \label{eq:1106}
\sum_{n=1}^{N} \frac{\expect{\sum_{k=0}^{K-1}  f_{n}(r_{n,k})\,T_{k} }}{\expect{\sum_{k=0}^{K-1} T_{k}}}
	\geq
\sum_{n=1}^{N} f_{n}\left( \frac{\expect{\sum_{k=0}^{K-1} r_{n,k} T_{k}}}{\expect{\sum_{k=0}^{K-1} T_{k}}}\right).
\end{equation}
Combining~\eqref{eq:1105}\eqref{eq:1106} and taking a $\limsup$ as $K\to\infty$ yields
\[
\begin{split}
& \limsup_{K\to\infty} \sum_{n=1}^{N} f_{n}\left( \frac{\expect{\sum_{k=0}^{K-1} r_{n,k} T_{k}}}{\expect{\sum_{k=0}^{K-1} T_{k}}}\right) \\
&\qquad \leq \frac{C\sum_{n=1}^{N}\lambda_{n}}{V} + \sum_{n=1}^{N} f_{n}(\oW_{n}^{*}).
\end{split}
\]
The next lemma, proved in Appendix~\ref{sec:2203}, completes the proof.

\begin{lem} \label{lem:1801}
If queues $\{Y_{n,k}\}_{k=0}^{\infty}$ are mean rate stable for all classes $n$, then
\[
\begin{split}
&\limsup_{K\to\infty} \sum_{n=1}^{N} f_{n}\left( \frac{\expect{\sum_{k=0}^{K-1} \sum_{i\in A_{n,k}} W_{n,k}^{(i)}}}{\expect{\sum_{k=0}^{K-1} \abs{A_{n,k}}}}\right) \\
& \quad\leq \limsup_{K\to\infty} \sum_{n=1}^{N} f_{n}\left( \frac{\expect{\sum_{k=0}^{K-1} r_{n,k} T_{k}}}{\expect{\sum_{k=0}^{K-1} T_{k}}}\right).
\end{split}
\]
\end{lem}
\end{proof}

\section{Delay-Constrained Optimal Power Control} \label{sec:powercontrol}

In this section we incorporate dynamic power control into the queueing system. As mentioned in Section~\ref{sec:702}, we focus on frame-based policies that allocate a constant power $P_{k} \in [\pmin, \pmax]$ over the duration of the $k$th busy period (we assume zero power is allocated when the system is idle). Here, interesting quantities such as frame size $T_{k}$, busy period $B_{k}$, the set $A_{n,k}$ of per-frame class $n$ arrivals, and queueing delay $W_{n,k}^{(i)}$ are all functions of power $P_{k}$. Similar to the delay definition~\eqref{eq:1103}, we define the average power consumption
\begin{equation} \label{eq:1701}
\Pav \triangleq \limsup_{K\to\infty} \frac{\expect{\sum_{k=0}^{K-1} P_{k}\, B_{k}(P_{k})}}{\expect{\sum_{k=0}^{K-1} T_{k}(P_{k})}},
\end{equation}
where $B_{k}(P_{k})$ and $T_{k}(P_{k})$ emphasize the power dependence of $B_{k}$ and $T_{k}$. It is easy to show that both $B_{k}(P_{k})$ and $T_{k}(P_{k})$ are decreasing in $P_{k}$. The goal is to solve the delay-constrained power minimization problem:
\begin{align}
\text{minimize:} &\quad \Pav \label{eq:701} \\
\text{subject to:} &\quad \oW_{n} \leq d_{n}, \quad \forall n\in\{1, \ldots, N\} \label{eq:702}
\end{align}
over frame-based power control and nonpreemptive priority policies. 

\subsection{Power-Delay Performance Region} \label{sec:1103}

Every frame-based power control and nonpreemptive priority policy can be viewed as a timing sharing or randomization of stationary policies that make the same deterministic decision in every frame. Using this point of view, next we give an example of the joint power-delay performance region resulting from frame-based policies. Consider a two-class nonpreemptive $M/G/1$ queue with parameters:
\begin{itemize}
\item $\lambda_{1} = 1$, $\lambda_{2} = 2$, $\expect{S_{1}} = \expect{S_{2}} = \expect{S_{2}^{2}} = 1$, $\expect{S_{1}^{2}} = 2$. $\mu(P) = P$. For each class $n\in\{1, 2\}$, the service time $X_{n}$ has mean $\expect{X_{n}} = \expect{S_{n}}/P$ and second moments $\expect{X_{n}^{2}} = \expect{S_{n}^{2}}/P^{2}$. For stability, we must have $\lambda_{1} \expect{X_{1}} + \lambda_{2} \expect{X_{2}} < 1 \Rightarrow P>3$. In this example, let $[4, 10]$ be the feasible power region.
\end{itemize}
Under a constant power allocation $P$, let $\mathcal{W}(P)$ denote the set of achievable queueing delay vectors $(\oW_{1}, \oW_{2})$. Define $\rho_{n} \triangleq \lambda_{n} \expect{X_{n}}$ and $R \triangleq \frac{1}{2} \sum_{n=1}^{2} \lambda_{n}\expect{X_{n}^{2}}$. Then we have
\begin{equation} \label{eq:1115}
\mathcal{W}(P) = \Set{ (\oW_{1}, \oW_{2}) |
	\begin{gathered}
	\oW_{n} \geq \frac{R}{1-\rho_{n}},\, n\in\{1, 2\} \\
	\sum_{n=1}^{2} \rho_{n} \oW_{n} = \frac{(\rho_{1}+\rho_{2}) R}{1-\rho_{1}-\rho_{2}}
	\end{gathered}
}.
\end{equation}
The inequalities in $\mathcal{W}(P)$ show that the minimum delay for each class is attained when it has priority over the other. The  equality in $\mathcal{W}(P)$ follows the $M/G/1$ conservation law~\cite{Kle64book}. Using the above parameters, we get
\[
\mathcal{W}(P) = \Set{ (\oW_{1}, \oW_{2}) |
	\begin{gathered}
	\oW_{1} \geq \frac{2}{P(P-1)} \\
	\oW_{2} \geq \frac{2}{P(P-2)} \\
	\oW_{1} + 2 \oW_{2} = \frac{6}{P(P-3)}
	\end{gathered}
}.
\]
Fig.~\ref{fig:601} shows the collection of delay regions $\mathcal{W}(P)$ for different values of $P\in [4, 10]$. This joint region contains all feasible delay vectors under  constant power allocations.
\begin{figure}[htb] 
\centering
\includegraphics[width=2in]{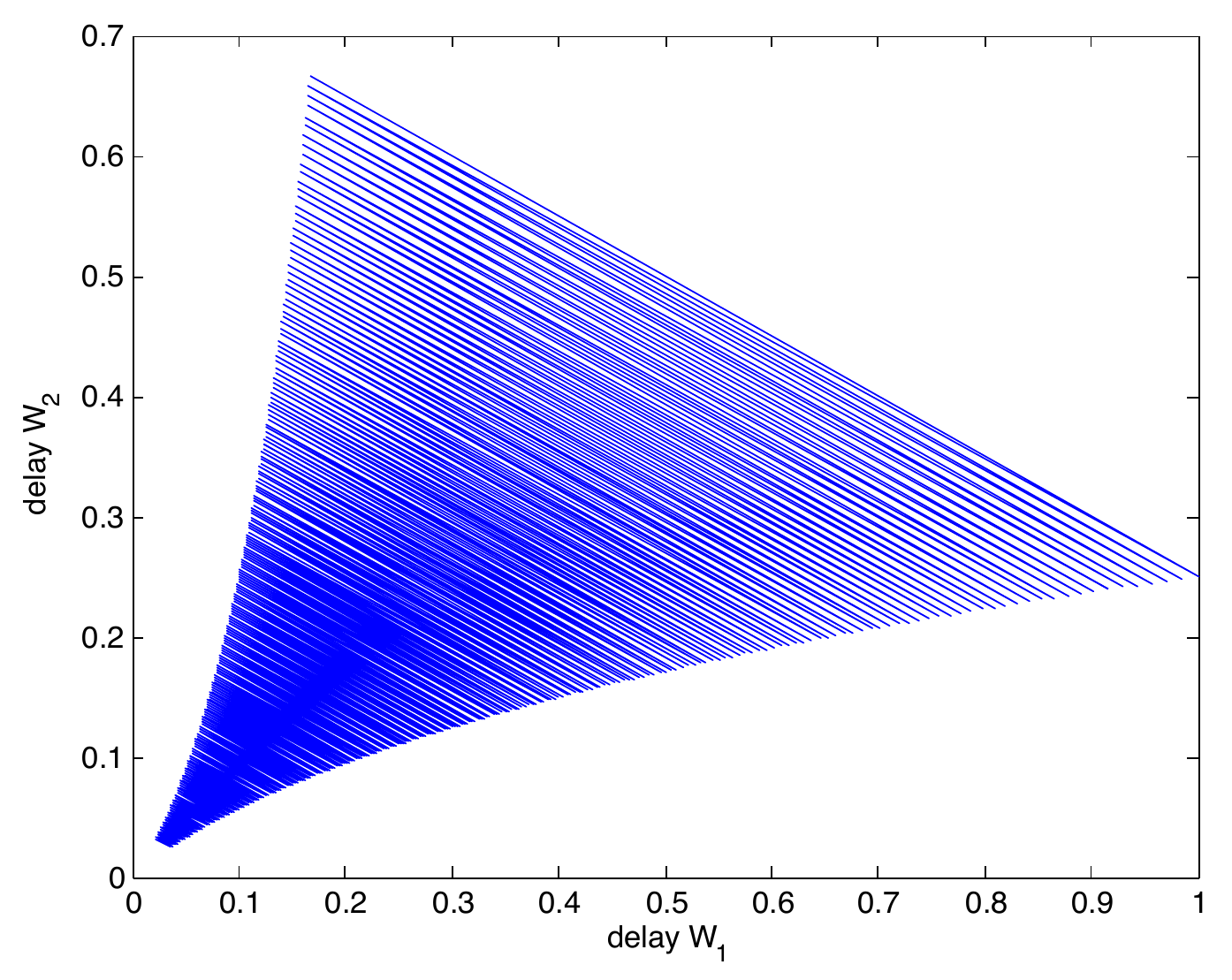}
\caption{The collection of average delay regions $\mathcal{W}(P)$ for different power levels $P\in[4, 10]$.}
\label{fig:601}
\end{figure}
\begin{figure}[htb] 
\centering
\includegraphics[width=2in]{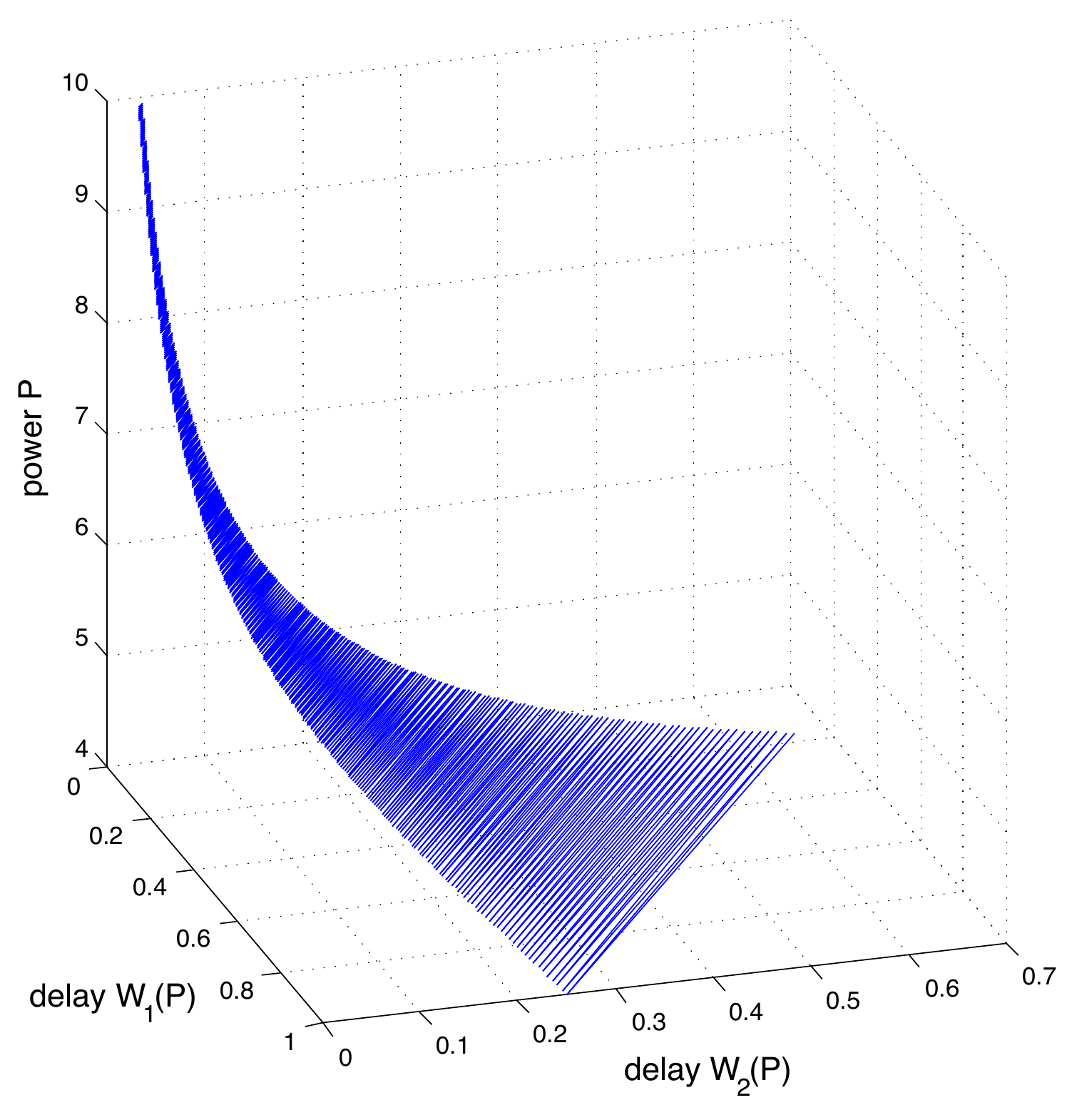}
\caption{The augmented performance region of power-delay vectors $(P, \oW_{1}(P), \oW_{2}(P))$.}
\label{fig:602}
\end{figure}
Fig.~\ref{fig:602} shows the associated augmented performance region of power-delay vectors $(P, \oW_{1}(P), \oW_{2}(P))$; its projection onto the delay plane is Fig.~\ref{fig:601}. After timing sharing or randomization, the performance region of all frame-based power control and nonpreemptive priority policies is the convex hull of Fig.~\ref{fig:602}. The problem~\eqref{eq:701}-\eqref{eq:702} is viewed a stochastic optimization over such a convexified power-delay performance region.

\subsection{Dynamic Power Control Policy} \label{sec:1102}
We setup the same virtual delay queues $Z_{n,k}$ as in~\eqref{eq:601}, and assume $Z_{n,0} = 0$ for all classes $n$. We represent a strict nonpreemptive priority policy by a permutation $(\pi_{n})_{n=1}^{N}$ of $\{1, \ldots, N\}$, where $\pi_{n}$ denotes the job class that gets the $n$th highest priority.

%\subsubsection*{Dynamic Power Control ($\mathsf{DynPower}$) Policy} 
\underline{\textit{Dynamic Power Control ($\mathsf{DynPower}$) Policy:}}

\begin{enumerate}
\item In the $k$th frame for each $k\in\Z^{+}$, use the nonpreemptive strict priority rule $(\pi_{n})_{n=1}^{N}$ that assigns priorities in the decreasing order of $Z_{n,k}/\expect{S_{n}}$; ties are broken arbitrarily. 
\item Allocate a fixed power $P_{k}$ in frame $k$, where $P_{k}$ is the solution to the following minimization of a weighted sum of power and average delays:
\begin{align}
\text{minimize:} &\quad \left(V\sum_{n=1}^{N} \lambda_{n} \expect{S_{n}}\right) \frac{P_{k}}{\mu(P_{k})} \label{eq:607} \\
&\qquad + \sum_{n=1}^{N} Z_{\pi_{n}, k} \, \lambda_{\pi_{n}} \oW_{\pi_{n}}(P_{k})  \notag \\
\text{subject to:} &\quad P_{k} \in [\pmin, \pmax].  \label{eq:608}
\end{align}
The value $\oW_{\pi_{n}}(P_{k})$, given later in~\eqref{eq:625}, is the average delay of class $\pi_{n}$ under the priority rule $(\pi_{n})_{n=1}^{N}$ and power allocation $P_{k}$. 
\item Update queues $Z_{n,k}$ for all classes $n\in\{1, \ldots, N\}$ by~\eqref{eq:601} at every frame boundary.
\end{enumerate}
The above $\mathsf{DynPower}$ policy requires the knowledge of arrival rates and the first two moments of job sizes for all classes $n$ (see~\eqref{eq:625}). We can remove its dependence on the second moments of job sizes, so that it only depends on the mean of arrivals and job sizes; see Appendix~\ref{sec:2201} for details.

%The $\mathsf{DynPower}$ policy is similar to the $\mathsf{DelayFeas}$ policy in the first delay feasibility problem, except that power is dynamically allocated in each frame by solving the one-dimensional optimization problem~\eqref{eq:607}-\eqref{eq:608}. Depending on the form of rate function $\mu(\cdot)$,  problem~\eqref{eq:607}-\eqref{eq:608} may not always be a convex program. For example,  it is not hard to show $\oW_{\pi_{n}}(P_{k})$ is convex in $P_{k}$ for all classes $n$, but $P_{k}/ \mu(P_{k})$ may be concave (e.g., take $\mu(P_{k}) = \sqrt{P_{k}}$). Yet, it is a one-dimensional optimization problem and may be solved using line search methods. 

\subsection{Motivation of the $\mathsf{DynPower}$ Policy}

We construct the  Lyapunov drift argument. Define the Lyapunov function $L(\bZ_{k}) = \frac{1}{2} \sum_{n=1}^{N} Z_{n,k}^{2}$ and the one-frame Lyapunov drift $\Delta(\bZ_{k}) = \expect{L(\bZ_{k+1}) - L(\bZ_{k}) \mid \bZ_{k}}$. Similar as the derivation in Section~\ref{sec:2202}, we have the Lyapunov drift inequality:
\begin{equation} \label{eq:624}
\Delta(\bZ_{k}) \leq C + \sum_{n=1}^{N} Z_{n,k}\, \expect{\sum_{i\in A_{n,k}} \paren{W_{n,k}^{(i)}-d_{n}} \mid \bZ_{k}}.
\end{equation}
Adding the weighted energy $V \expect{P_{k}\, B_{k}(P_{k}) \mid \bZ_{k}}$ to both sides of~\eqref{eq:624}, where $V>0$ is a control parameter, yields
\begin{equation} \label{eq:603}
\Delta(\bZ_{k}) + V \expect{P_{k}\, B_{k}(P_{k}) \mid \bZ_{k}} \leq C + \Phi(\bZ_{k}),
\end{equation}
where
\[
\begin{split}
\Phi(\bZ_{k}) &\triangleq \mathbb{E} \Bigg[ V P_{k}\, B_{k}(P_{k}) \\
&\quad +\sum_{n=1}^{N} Z_{n,k} \sum_{i\in A_{n,k}} (W_{n,k}^{(i)}-d_{n})\mid \bZ_{k} \Bigg].
\end{split}
\]
We are interested in the frame-based policy that, in each frame $k$,  allocates power and assigns priorities to minimize the ratio
\begin{equation} \label{eq:604}
\frac{
	\Phi(\bZ_{k})
}{
	\expect{T_{k}(P_{k}) \mid \bZ_{k}}
}.
\end{equation}
Note that frame size $T_{k}(P_{k})$ depends on $\bZ_{k}$ because the power allocation that affects $T_{k}(P_{k})$ may be $\bZ_{k}$-dependent. For any given power allocation $P_{k}$, $T_{k}(P_{k})$ is independent of $\bZ_{k}$.

Lemma~\ref{lem:607} next shows that the minimizer of~\eqref{eq:604} is a deterministic power allocation and strict nonpreemptive priority policy. Specifically, we may consider each $p\in \mathcal{P}$ in Lemma~\ref{lem:607}  denotes a deterministic power allocation and strict  priority policy, and random variable $P$ denotes a randomized power control and priority policy.

\begin{lem} \label{lem:607}
Let $P$ be a continuous random variable with state space $\mathcal{P}$. Let $G$ and $H$ be two random variables that depend on $P$ such that, for each $p \in \mathcal{P}$, $G(p)$ and $H(p)$ are well-defined random variables. Define
\[
p^{*} \triangleq \argmin_{p\in\mathcal{P}} \frac{\expect{G(p)}}{\expect{H(p)}},\quad U^{*} \triangleq \frac{\expect{G(p^{*})}}{\expect{H(p^{*})}}.
\]
Then $\frac{\expect{G}}{\expect{H}} \geq U^{*}$ regardless of the  distribution of $P$.
\end{lem}
\begin{proof}
For each $p\in \mathcal{P}$, we have $\frac{\expect{G(p)}}{\expect{H(p)}} \geq U^{*}$. Then
\[
\frac{\expect{G}}{\expect{H}} = \frac{\mathbb{E}_{P}\left[ \expect{G(p)} \right]}{\mathbb{E}_{P}\left[ \expect{H(p)} \right]} \geq \frac{ \mathbb{E}_{P}\left[ U^{*} \expect{H(p)} \right]}{\mathbb{E}_{P}\left[ \expect{H(p)} \right]} = U^{*},
\]
which is independent of the distribution of $P$.
\end{proof}

Under a fixed power allocation $P_{k}$ and a strict nonpreemptive priority rule,~\eqref{eq:604} is equal to
\begin{equation} \label{eq:1604}
\begin{split}
&\frac{
	V \expect{P_{k}B_{k}(P_{k})} + \sum_{n=1}^{N} Z_{n,k} \, \lambda_{n} (\oW_{n,k}(P_{k}) - d_{n}) \expect{T_{k}(P_{k})}
	}{\expect{T_{k}(P_{k})}} \\
&\quad = V P_{k} \frac{\sum_{n=1}^{N} \lambda_{n} \expect{S_{n}}}{\mu(P_{k})} + \sum_{n=1}^{N} Z_{n,k} \, \lambda_{n} (\oW_{n,k}(P_{k})-d_{n}),
\end{split}
\end{equation}
where by renewal theory
\[
\frac{\expect{B_{k}(P_{k})}}{\expect{T_{k}(P_{k})}} = \sum_{n=1}^{N} \rho_{n}(P_{k}) = \sum_{n=1}^{N} \lambda_{n} \frac{\expect{S_{n}}}{\mu(P_{k})}
\]
and power-dependent terms are written as functions of $P_{k}$. It follows that our desired policy in every frame $k$ minimizes
\begin{equation} \label{eq:606}
\left(V\sum_{n=1}^{N}\lambda_{n}\expect{S_{n}}\right) \frac{P_{k}}{\mu(P_{k})} + \sum_{n=1}^{N} Z_{n,k} \, \lambda_{n} \oW_{n,k}(P_{k})
\end{equation}
over constant power allocations $P_{k}\in [\pmin, \pmax]$ and nonpreemptive strict priority rules.

To further simplify, for each fixed power level $P_{k}$, by Lemma~\ref{lem:603}, the $c\mu$ rule that assigns priorities in the decreasing order of $Z_{n,k}/\expect{S_{n}}$ minimizes the second term of~\eqref{eq:606} (note that minimizing a linear function over strict priority rules is equivalent to minimizing over all randomized priority rules, since a vertex of the performance polytope attains the minimum). This strict priority policy is optimal regardless of the value of $P_{k}$, and thus is overall optimal; priority assignment and power control are decoupled. We represent the optimal priority policy by $(\pi_{n})_{n=1}^{N}$, recalling that $\pi_{n}$ denotes the job class that gets the $n$th highest priority. Under priorities $(\pi_{n})_{n=1}^{N}$ and a fixed power allocation $P_{k}$,  the average delay $\oW_{\pi_{n}}(P_{k})$ for class $\pi_{n}$ is equal to
\begin{equation} \label{eq:625}
\begin{split}
\oW_{\pi_{n}}(P_{k}) &= \frac{\frac{1}{2} \sum_{n=1}^{N} \lambda_{n} \expect{X_{n}^{2}}}{(1-\sum_{m=0}^{n-1} \rho_{\pi_{m}})(1-\sum_{m=0}^{n} \rho_{\pi_{m}})} \\
&= \frac{
	\frac{1}{2} \sum_{n=1}^{N} \lambda_{n} \expect{S_{n}^{2}}
}{
	(\mu(P_{k}) - \sum_{m=0}^{n-1} \hat{\rho}_{\pi_{m}})(\mu(P_{k}) - \sum_{m=0}^{n} \hat{\rho}_{\pi_{m}})
},
\end{split}
\end{equation}
where  $\hat{\rho}_{\pi_{m}} \triangleq \lambda_{\pi_{m}} \expect{S_{\pi_{m}}}$ if $m\geq 1$ and $0$ if $m=0$. The above discussions lead to the $\mathsf{DynPower}$ policy.

\subsection{Performance of the $\mathsf{DynPower}$ Policy}

\begin{thm} \label{thm:603}
Let $P^{*}$ be the optimal average power of the problem~\eqref{eq:701}-\eqref{eq:702}. The $\mathsf{DynPower}$ policy achieves delay constraints $\oW_{n}\leq d_{n}$ for all classes $n\in\{1, \ldots, N\}$ and attains average power $ \Pav$ satisfying
\[
\Pav \leq \frac{C\sum_{n=1}^{N}\lambda_{n}}{V} + P^{*},
\]
where $C>0$ is a finite constant and $V>0$ a predefined control parameter.
\end{thm}

\begin{proof}[Proof of Theorem~\ref{thm:603}]
As discussed in Section~\ref{sec:1103}, the power-delay performance region in this problem is spanned by stationary power control and nonpreemptive priority policies that use the same (possibly random) decision in every frame. Let $\pi^{*}$ denote one such policy that yields the optimal average power $P^{*}$ with feasible delays $\oW_{n}^{*} \leq d_{n}$ for all classes $n$. Let $P_{k}^{*}$ be its power allocation in frame $k$. Since policy $\pi^{*}$ makes i.i.d. decisions over frames, by renewal reward theory we have
\[
P^{*} = \frac{\expect{P_{k}^{*} B(P_{k}^{*})}}{\expect{T(P_{k}^{*})}}.
\]
Then the ratio $\frac{\Phi(\bZ_{k})}{\expect{T_{k}(P_{k}) \mid \bZ_{k}}}$ under policy $\pi^{*}$ (see the left side of~\eqref{eq:1604}) is equal to
\[
V\frac{\expect{P_{k}^{*} B(P_{k}^{*})}}{\expect{T(P_{k}^{*})}} + \sum_{n=1}^{N} Z_{n,k}\, \lambda_{n} \left(\oW_{n}^{*}-d_{n}\right) \leq V P^{*}.
\]
Since the $\mathsf{DynPower}$ policy minimizes $\frac{\Phi(\bZ_{k})}{\expect{T_{k}(P_{k}) \mid \bZ_{k}}}$ over frame-based policies, including the optimal policy $\pi^{*}$, the ratio $\frac{\Phi(\bZ_{k})}{\expect{T_{k}(P_{k}) \mid \bZ_{k}}}$ under the $\mathsf{DynPower}$ policy satisfies
\[
\frac{\Phi(\bZ_{k})}{\expect{T_{k}(P_{k}) \mid \bZ_{k}}} \leq V P^{*} \Rightarrow \Phi(\bZ_{k}) \leq V P^{*} \expect{T_{k}(P_{k}) \mid \bZ_{k}}.
\]
Using this bound in~\eqref{eq:603} yields
\[
\Delta(\bZ_{k}) + V \expect{P_{k}\, B_{k}(P_{k}) \mid \bZ_{k}} \leq C+ V P^{*} \, \expect{T_{k}(P_{k}) \mid \bZ_{k}}.
\]
Taking expectation, summing over $k\in\{0, \ldots, K-1\}$, and noting $L(\bZ_{0})=0$ yields
\begin{equation} \label{eq:610}
\begin{split}
&		\expect{L(\bZ_{K})} + V \sum_{k=0}^{K-1} \expect{P_{k}\,B_{k}(P_{k})} \\
&\quad	\leq KC + V P^{*}\, \expect{\sum_{k=0}^{K-1} T_{k}(P_{k})}.
\end{split}
\end{equation}
Since $\expect{T_{k}(P_{k})}$ is decreasing in $P_{k}$ and, under a fixed power allocation, is independent of scheduling policies, we get $\expect{T_{k}(P_{k})} \leq \expect{T_{0}(\pmin)}$ and
\[
\begin{split}
&		\expect{L(\bZ_{K})} + V \sum_{k=0}^{K-1} \expect{P_{k}\,B_{k}(P_{k})} \\
&\quad	\leq K( C + V P^{*}\, \expect{T_{0}(\pmin)}).
\end{split}
\]
Removing the second term and dividing by $K^{2}$ yields
\[
\frac{\expect{L(\bZ_{K})}}{K^{2}} \leq \frac{C+V  P^{*}\,  \expect{T_{0}(\pmin)}}{K}.
\]
Combining it with
\[
0\leq \frac{\expect{Z_{n,K}}}{K} \leq \sqrt{\frac{\expect{Z_{n,K}^{2}}}{K^{2}}} \leq \sqrt{ \frac{2 \expect{L(\bZ_{K})}}{K^{2}}}
\]
and passing $K\to\infty$ proves that queue $\{Z_{n,k}\}_{k=0}^{\infty}$ is mean rate stable for all classes $n$. Thus $\oW_{n}\leq d_{n}$ for all $n$ by Lemma~\ref{lem:602}.

Further, removing the first term in~\eqref{eq:610} and dividing the result by $V \expect{ \sum_{k=0}^{K-1} T_{k}(P_{k})}$ yields
\[
\begin{split}
\frac{ \expect{\sum_{k=0}^{K-1} P_{k}\, B_{k}(P_{k})}}{\expect{ \sum_{k=0}^{K-1} T_{k}(P_{k})}}
&\leq \frac{C}{V}\frac{K}{\expect{ \sum_{k=0}^{K-1} T_{k}(P_{k})}} + P^{*} \\
&\stackrel{(a)}{\leq} \frac{C\sum_{n=1}^{N}\lambda_{n}}{V} + P^{*},
\end{split}
\]
where (a) uses $\expect{T_{k}(P_{k})} \geq \expect{I_{k}} = 1/(\sum_{n=1}^{N} \lambda_{n})$. Passing $K\to\infty$ completes the proof.
\end{proof}

\section{Optimizing Delay Penalties with Average Power Constraint} \label{sec:1801}

The fourth problem we consider is to, over frame-based power control and nonpreemptive priority policies, minimize a separable convex function of delay vectors $(\oW_{n})_{n=1}^{N}$ subject to an average power constraint:
\begin{align}
\text{minimize:} &\quad \sum_{n=1}^{N} f_{n}(\oW_{n}) \label{eq:1705} \\
\text{subject to:} &\quad \Pav \leq \pconst. \label{eq:1706}
\end{align}
The value $\Pav$ is defined in~\eqref{eq:1701} and $\pconst >0$ is a given feasible bound. The penalty functions $f_{n}(\cdot)$  are assumed nondecreasing, nonnegative, continuous, and convex for all classes $n$. Power allocation in every busy period takes values in $[\pmin, \pmax]$, and no power is allocated when the system is idle. In this problem, the region of feasible power-delay vectors $(\Pav, \oW_{1}, \ldots, \oW_{N})$ is complicated because feasible delays $(\oW_{n})_{n=1}^{N}$ are indirectly decided by the power constraint~\eqref{eq:1706}. Using the same methodology as in the previous three problems, we construct a frame-based policy to solve~\eqref{eq:1705}-\eqref{eq:1706}.

We setup the virtual delay queue $\{Y_{n,k}\}_{k=0}^{\infty}$ for each class $n\in\{1, \ldots, N\}$ as in~\eqref{eq:622}, in which the auxiliary variable $r_{n,k}$ takes values in $[0, \rmax_{n}]$ for some $\rmax_{n}>0$ sufficiently large.\footnote{For each class $n$, we need $\rmax_{n}$ to be larger than the optimal delay $\oW_{n}^{*}$ in problem~\eqref{eq:1705}-\eqref{eq:1706}. One way is to let $\rmax_{n}$ be the maximum average delay over all classes under the minimum power allocation $\pmin$.}  Define the discrete-time \emph{virtual power queue} $\{X_{k}\}_{k=0}^{\infty}$ that evolves at frame boundaries $\{t_{k}\}_{k=0}^{\infty}$ as
\begin{equation} \label{eq:1702}
X_{k+1} = \max\left[ X_{k} + P_{k} B_{k}(P_{k}) - \pconst T_{k}(P_{k}), \, 0\right].
\end{equation}
Assume $X_{0} = 0$. The $\{X_{k}\}_{k=0}^{\infty}$ queue helps to achieve the power constraint $\Pav \leq \pconst$.
\begin{lem} \label{lem:1701}
If the virtual power queue $\{X_{k}\}_{k=0}^{\infty}$ is mean rate stable, then $\Pav \leq \pconst$.
\end{lem}
\begin{proof}
Given in Appendix~\ref{sec:2401}.
\end{proof}

\subsection{Power-constrained Delay Fairness Policy} \label{sec:1805}

%\subsubsection*{Power-constrained Delay Fairness $(\mathsf{PwDelayFair})$ Policy:} 

\underline{\textit{Power-constrained Delay Fairness $(\mathsf{PwDelayFair})$ Policy:}}

In the busy period of each frame $k\in\Z^{+}$, after observing $X_{k}$ and $(Y_{n,k})_{n=1}^{N}$:
\begin{enumerate}
\item Use the nonpreemptive strict priority rule $(\pi_{n})_{n=1}^{N}$ that assigns priorities in the decreasing order of $Y_{n,k}/\expect{S_{n}}$; ties are broken arbitrarily. 
\item Allocate power $P_{k}$ for the duration of the busy period, where $P_{k}$ solves:
\begin{align*}
\text{minimize:} &\quad X_{k} \left[ - \pconst + \frac{P_{k}}{\mu(P_{k})} \sum_{n=1}^{N} \lambda_{n} \expect{S_{n}} \right]  \\
&\qquad + \sum_{n=1}^{N} Y_{\pi_{n},k} \, \lambda_{\pi_{n}} \oW_{\pi_{n}}(P_{k})  \\
\text{subject to:} &\quad P_{k} \in [\pmin, \pmax], 
\end{align*}
where $\oW_{\pi_{n}}(P_{k})$ is defined in~\eqref{eq:625}.
\item Update $X_{k}$ and $Y_{n,k}$ for all classes $n$ at every frame boundary by~\eqref{eq:1702} and~\eqref{eq:622}, respectively. In~\eqref{eq:622}, the auxiliary variable $r_{n,k}$ is the solution to
\begin{align*}
\text{minimize:}		&\quad V f_{n}(r_{n,k}) - Y_{n,k} \, \lambda_{n}\, r_{n,k} \\
\text{subject to:}	&\quad 0\leq r_{n,k}\leq \rmax_{n}.
\end{align*}
\end{enumerate}

\subsection{Motivation of the $\mathsf{PwDelayFair}$ Policy}
The construction of the Lyapunov drift argument follows closely with those in the previous problems; details are omitted for brevity. Define vector $\chi_{k} = [X_{k}; Y_{1,k}, \ldots, Y_{N,k}]$, the Lyapunov function $L(\chi_{k}) \triangleq \frac{1}{2} (X_{k}^{2} + \sum_{n=1}^{N} Y_{n,k}^{2})$,  and the one-frame Lyapunov drift $\Delta(\chi_{k}) \triangleq \expect{ L(\chi_{k+1}) - L(\chi_{k}) \mid \chi_{k}}$. We can show there exists a finite constant $C>0$ such that
\begin{equation} \label{eq:2205}
\begin{split}
\Delta(\chi_{k}) &\leq C + X_{k} \expect{P_{k}B_{k}(P_{k}) - \pconst \, T_{k}(P_{k}) \mid \chi_{k} } \\
&\quad  + \sum_{n=1}^{N} Y_{n,k} \, \expect{\sum_{i\in A_{n,k}} \left(W_{n,k}^{(i)}\ - r_{n,k} \right) \mid \chi_{k} }.
\end{split}
\end{equation}
Adding the term $V \sum_{n=1}^{N} \expect{f_{n}(r_{n,k}) \, T_{k}(P_{k}) \mid \chi_{k}}$  to both sides of~\eqref{eq:2205}, where $V>0$ is a control parameter, and evaluating the result under a frame-based policy yields
\begin{equation} \label{eq:1707}
\Delta(\chi_{k}) + V \sum_{n=1}^{N} \expect{f_{n}(r_{n,k})\, T_{k}(P_{k}) \mid \chi_{k}} \leq C + \Psi(\chi_{k}),
\end{equation}
where
\[
\begin{split}
&\Psi(\chi_{k}) \triangleq \expect{T_{k}(P_{k}) \mid \chi_{k}} \sum_{n=1}^{N} Y_{n,k} \, \lambda_{n} \oW_{n,k}(P_{k})  \\
&\quad + X_{k} \expect{P_{k} B_{k}(P_{k}) \mid \chi_{k}} - X_{k} \pconst \, \expect{T_{k}(P_{k}) \mid \chi_{k}} \\
&\quad + \expect{T_{k}(P_{k}) \mid \chi_{k}} \sum_{n=1}^{N} \expect{V f_{n}(r_{n,k}) - Y_{n,k} \, \lambda_{n}\, r_{n,k} \mid \chi_{k}},
\end{split}
\]
where $\oW_{n,k}(P_{k})$ is the average delay of class $n$ if the control  and power allocation in frame $k$ is repeated in every frame.

We are interested in the frame-based policy that minimizes the ratio $\frac{\Psi(\chi_{k})}{\expect{T_{k}(P_{k}) \mid \chi_{k}}}$ in each frame $k\in \Z^{+}$. Lemma~\ref{lem:607} shows the minimizer is a deterministic policy, under which the ratio is equal to
\[
\begin{split}
& \sum_{n=1}^{N} Y_{n,k} \, \lambda_{n} \oW_{n,k}(P_{k}) + X_{k} \left( P_{k}\, \rho_{\text{sum}}(P_{k}) - \pconst \right) \\
& + \sum_{n=1}^{N}  \left( V f_{n}(r_{n,k}) - Y_{n,k} \, \lambda_{n} \, r_{n,k} \right),
\end{split}
\]
where $\rho_{\text{sum}}(P_{k}) \triangleq \sum_{n=1}^{N} \lambda_{n} \expect{S_{n}} / \mu(P_{k})$. Under similar simplifications as the $\mathsf{DynPower}$ policy in Section~\ref{sec:1102}, we can show that the $\mathsf{PwDelayFair}$ policy is the desired policy.

\subsection{Performance of the $\mathsf{PwDelayFair}$ Policy}

\begin{thm} \label{thm:1801}
For any feasible average power constraint $\Pav\leq \pconst$,  the $\mathsf{PwDelayFair}$ policy satisfies $\Pav\leq \pconst$ and yields average delay penalty satisfying
\begin{equation} \label{eq:1805}
\begin{split}
& \limsup_{K\to\infty}\sum_{n=1}^{N} f_{n}\left( \frac{\expect{\sum_{k=0}^{K-1} \sum_{i\in A_{n,k}} W_{n,k}^{(i)}}}{\expect{\sum_{k=0}^{K-1} \abs{A_{n,k}}}}\right)  \\
&\quad  \leq \frac{C \sum_{n=1}^{N} \lambda_{n}}{V} + \sum_{n=1}^{N} f_{n}(\oW_{n}^{*}),
\end{split}
\end{equation}
where $V>0$ is a predefined control parameter.
\end{thm}

\begin{proof}[Proof of Theorem~\ref{thm:1801}]
Let $\pirand^{*}$ be the frame-based randomized policy that solves~\eqref{eq:1705}-\eqref{eq:1706}. Let $(\oW_{n}^{*})_{n=1}^{N}$ be the optimal average delay vector, and $\Pav^{*}$, where $\Pav^{*}\leq \pconst$, be the associated power consumption. In frame $k\in\Z^{+}$, the ratio $\frac{\Psi(\chi_{k})}{\expect{T_{k}(P_{k}) \mid \chi_{k}}}$ evaluated under policy $\pirand^{*}$ and  genie decisions $r_{n,k}^{*} = \oW_{n}^{*}$ for all classes $n$ is equal to
\begin{equation} \label{eq:1708}
\begin{split}
&\sum_{n=1}^{N} Y_{n,k} \, \lambda_{n} \oW_{n}^{*} + X_{k} \Pav^{*} - X_{k} \pconst \\
&+ \sum_{n=1}^{N} \left( V f_{n}(\oW_{n}^{*}) - Y_{n,k} \, \lambda_{n} \oW_{n}^{*}\right) \leq V \sum_{n=1}^{N} f_{n}(\oW_{n}^{*}).
\end{split}
\end{equation}
Since the $\mathsf{PwDelayFair}$ policy minimizes $\frac{\Psi(\chi_{k})}{\expect{T_{k}(P_{k}) \mid \chi_{k}}}$ in every frame $k$, the ratio under the $\mathsf{PwDelayFair}$ policy satisfies
\[
\frac{\Psi(\chi_{k})}{\expect{T_{k}(P_{k}) \mid \chi_{k}}} \leq V \sum_{n=1}^{N} f_{n}(\oW_{n}^{*}).
\]
Then~\eqref{eq:1707} under the $\mathsf{PwDelayFair}$ policy satisfies
\begin{equation} \label{eq:1803}
\begin{split}
&\Delta(\chi_{k}) + V \expect{\sum_{n=1}^{N} f_{n}(r_{n,k})\, T_{k}(P_{k}) \mid \chi_{k}} \\
&\quad \leq C + V \expect{T_{k}(P_{k}) \mid \chi_{k}} \sum_{n=1}^{N} f_{n}(\oW_{n}^{*}).
\end{split}
\end{equation}

Removing the second term in~\eqref{eq:1803} and taking expectation, we get
\[
\expect{L(\chi_{k+1})} - \expect{L(\chi_{k})} \leq C + V\expect{T_{k}(P_{k})} \sum_{n=1}^{N} f_{n}(\oW_{n}^{*}).
\]
Summing over $k\in\{0, \ldots, K-1\}$, and using $L(\chi_{0})=0$ yields
\begin{equation} \label{eq:1804}
\expect{L(\chi_{K})} \leq KC + V \expect{\sum_{k=0}^{K-1} T_{k}(P_{k})} \sum_{n=1}^{N} f_{n}(\oW_{n}^{*}) \leq K C_{1}
\end{equation}
where $C_{1} \triangleq C + V \expect{T_{0}(\pmin)} \sum_{n=1}^{N} f_{n}(\oW_{n}^{*})$, and we have used $\expect{T_{k}(P_{k})}\leq \expect{T_{0}(\pmin)}$. Inequality~\eqref{eq:1804} suffices to conclude that queues $X_{k}$ and $Y_{n,k}$ for all classes $n$ are all mean rate stable. From Lemma~\ref{lem:1701} the constraint $\Pav \leq \pconst$ is achieved.  The proof of~\eqref{eq:1805} follows that of Theorem~\ref{thm:602}.
\end{proof}

\section{Simulations} \label{sec:2402}
Here we simulate the $\mathsf{DelayFeas}$ and $\mathsf{DelayFair}$ policy in the first two delay control problems; simulations for the $\mathsf{DynPower}$ and $\mathsf{PwDelayFair}$ policy in the last two delay-power control problems are our future work. The setup is as follows. Consider a two-class $M/M/1$ queue with Poisson arrival rates $(\lambda_{1}, \lambda_{2}) = (1, 2)$, loading factors $(\rho_{1}, \rho_{2}) = (0.4, 0.4)$, and mean exponential service times $\expect{X_{1}} = \rho_{1}/\lambda_{1} = 0.4$ and $\expect{X_{2}} = \rho_{2}/\lambda_{2}= 0.2$ (we use service times directly since there is no power control). The average delay region $\mathcal{W}$ of this two-class $M/M/1$ queue, given in~\eqref{eq:1115}, is
\begin{equation} \label{eq:1119}
\mathcal{W} = \Set{ (\oW_{1}, \oW_{2}) | \begin{gathered} \oW_{1} + \oW_{2} = 2.4 \\ \oW_{1} \geq 0.4, \, \oW_{2} \geq 0.4 \end{gathered}}.
\end{equation}

For the $\mathsf{DelayFeas}$ policy, we consider five sets of delay constraints $(d_{1}, d_{2}) = (0.45, 2.05)$, $(0.85, 1.65)$, $(1.25, 1.25)$, $(1.65, 0.85)$, and $(2.05, 0.45)$; they are all $(0.05, 0.05)$ away from a feasible point on $\mathcal{W}$. For each constraint set $(d_{1}, d_{2})$, we repeat the simulation for $10$ times and take an average on the resulting average delay, where each simulation is run for $10^{6}$ frames. The results are given in Fig.~\ref{fig:1901}, which shows that the $\mathsf{DelayFeas}$ policy adaptively yields feasible average delays in response to different constraints.
\begin{figure}[htb]
\centering
\includegraphics[width=3in]{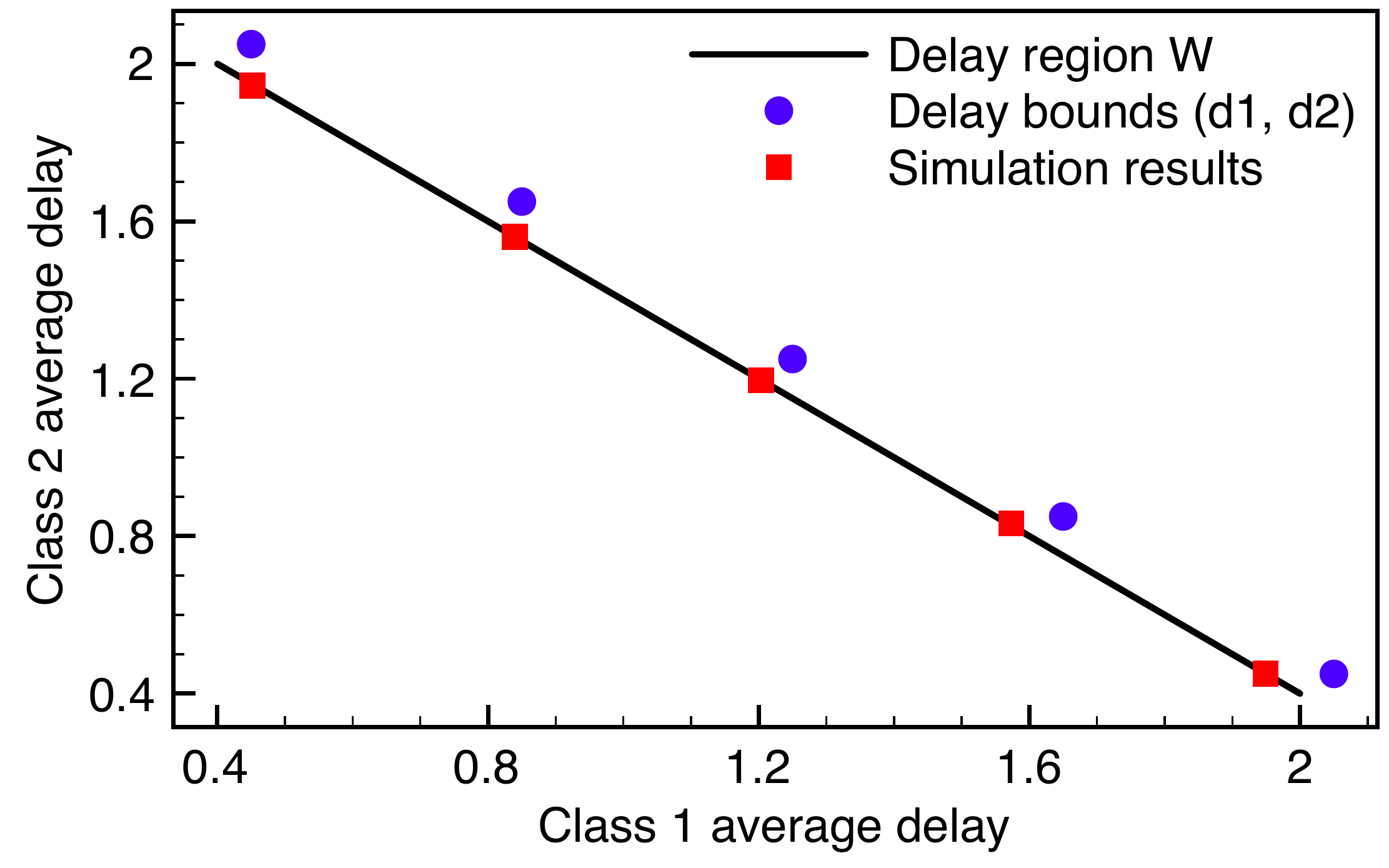}
\caption{The performance of the $\mathsf{DelayFeas}$ policy under different delay constraints $(d_{1}, d_{2})$.}
\label{fig:1901}
\end{figure}

Next, for the $\mathsf{DelayFair}$ policy, we consider the following delay proportional fairness problem:
\begin{align}
\text{minimize:} &\quad \frac{1}{2} \oW_{1}^{2} + 2 \oW_{2}^{2}  \label{eq:1116} \\
\text{subject to:} &\quad (\oW_{1}, \oW_{2}) \in \mathcal{W} \label{eq:1117} \\
&\quad \oW_{1} \leq 2, \oW_{2} \leq 2 \label{eq:1901}
\end{align}
where the delay region $\mathcal{W}$ is given in~\eqref{eq:1119}. The additional delay constraints~\eqref{eq:1901} are chosen to be non-restrictive for the ease of demonstration. The optimal solution to~\eqref{eq:1116}-\eqref{eq:1901} is $(\oW_{1}^{*}, \oW_{2}^{*}) = (1.92, 0.48)$; the optimal delay penalty is $\frac{1}{2} (\oW_{1}^{*})^{2} + 2 (\oW_{2}^{*})^{2} =2.304$. We simulate the $\mathsf{DelayFair}$ policy for different values of control parameter $V\in\{10^{2}, 10^{3}, 5\times 10^{3}, 10^{4}\}$.  The results are in Table~\ref{table:1102}.
\begin{table}[htbp]
\centering
\begin{tabular}{cccc}
$V$ & $\oW_{1}^{\mathsf{DelayFair}}$ & $\oW_{2}^{\mathsf{DelayFair}}$ & Delay penalty \\ \hline
$100$ & $1.611$ & $0.785$ & $2.529$ \\
$1000$ & $1.809$ & $0.591$ & $2.335$ \\
$5000$ & $1.879$ & $0.523$ & $2.312$ \\
$10000$ & $1.894$ & $0.503$ & $2.301$ \\ \hline
Optimal value: & $1.92$ & $0.48$ & $2.304$
\end{tabular}
\caption{The average delays and delay penalty under the $\mathsf{DelayFair}$ policy for different values of control parameter $V$.}
\label{table:1102}
\end{table}
Every entry in Table~\ref{table:1102} is the average over $10$ simulation runs, where each simulation is run for $10^{6}$ frames. As $V$ increases, the $\mathsf{DelayFair}$ policy yields average delays approaching the optimal $(1.92, 0.48)$ and the optimal penalty $2.304$.

\section{Conclusions}
This paper solves constrained delay-power stochastic optimization problems in a nonpreemptive multi-class $M/G/1$ queue from a new mathematical programming perspective. After characterizing the performance region by the collection of all frame-based randomizations of \emph{base policies} that comprise deterministic power control and nonpreemptive strict priority policies, we use the Lyapunov optimization theory to construct dynamic control algorithms that yields near-optimal performance. These policies greedily select and run a base policy in every frame by minimizing a ratio of an expected ``drift plus penalty'' sum over the expected frame size, and require limited statistical knowledge of the system. Time average constraints are turned into virtual queues that need to be stabilized.

While this paper studies  delay and power control in a nonpreemptive multi-class $M/G/1$ queue, our framework shall have a much wider applicability to other stochastic optimization problems over queueing networks, especially those that satisfy strong (possibly generalized~\cite{BaN96}) conservation laws and have polymatroidal performance regions. Different performance metrics such as throughput (together with admission control), delay, power, and functions of them can be mixed together to serve as objective functions or time average constraints. It is of interest to us to explore all these directions.

Another connection is, in~\cite{LaN10arXivb}, we have used the frame-based Lyapunov optimization theory to optimize a general functional objective over an inner bound on the performance region of a restless bandit problem with Markov ON/OFF bandits. This inner bound approach can be viewed as an approximation to such complex restless bandit problems. Multi-class queueing systems and restless bandit problems are two prominent examples of general stochastic control problems. Thus it would be interesting to develop the Lyapunov optimization theory as a unified framework to attack other open stochastic control problems.

%\section{Acknowledgments}
%This material is supported in part  by one or more of the following: the DARPA IT-MANET program grant W911NF-07-0028, the NSF Career grant CCF-0747525, and continuing through participation in the Network Science Collaborative Technology Alliance sponsored by the U.S. Army Research Laboratory.

%\appendices
%\section{} \label{}
%\bibliographystyle{abbrv}
\bibliographystyle{IEEEtran}
\bibliography{/Users/chihping/Desktop/bibliography/IEEEabrv,/Users/chihping/Desktop/bibliography/myabrv,/Users/chihping/Desktop/bibliography/mypaperbib}

\appendices

\section{} \label{sec:2204}

\begin{proof}[Proof of Lemma~\ref{lem:2201}]
We index all $N!$ nonpreemptive strict priority policies by $\{\pi_{j}\}_{j\in \mathcal{J}}$, where $\mathcal{J} = \{1, \ldots, N!\}$ is an index set and $\pi_{j}$ denotes the $j$th priority ordering. Consider a randomized policy $\pirand$ defined by the probability distribution $\{\alpha_{j}\}_{j\in\mathcal{J}}$, where $\pirand$ uses priority ordering $\pi_{j}$ for the duration of a frame with probability $\alpha_{j}$ in every frame. Let $W_{n}^{\text{sum}}(\pi_{j})$ denote the sum of queueing delays in class $n$ during a frame in which policy $\pi_{j}$ is used. Likewise, define $W_{n}^{\text{sum}}(\pirand)$ under policy $\pirand$. By conditional expectation, we get
\begin{equation} \label{eq:2210}
\expect{W_{n}^{\text{sum}}(\pirand)} = \sum_{j\in\mathcal{J}} \alpha_{j}\, \expect{W_{n}^{\text{sum}}(\pi_{j})}.
\end{equation}
Next, define $\oW_{n}(\pi_{j})$ as the average queueing delay for class $n$ if policy $\pi_{j}$ is used in every frame. Define $\oW_{n}(\pirand)$ under policy $\pirand$ similarly. From renewal reward theory, we have
\begin{align}
\expect{W_{n}^{\text{sum}}(\pirand)} &= \lambda_{n} \oW_{n}(\pirand) \expect{T} \label{eq:2208} \\
\expect{W_{n}^{\text{sum}}(\pi_{j})} &= \lambda_{n} \oW_{n}(\pi_{j}) \expect{T}, \label{eq:2209}
\end{align}
where $\expect{T}$ is the average frame size. Note that $\expect{T}$ is independent of scheduling policies.  From~\eqref{eq:2210}-\eqref{eq:2209} we get
\begin{equation}\label{eq:2212}
\oW_{n}(\pirand) =  \sum_{j\in\mathcal{J}} \alpha_{j}\, \oW_{n}(\pi_{j}).
\end{equation}
Define $x_{n}(\pi_{j}) \triangleq \rho_{n} \oW_{n}(\pi_{j})$ for all priority orderings $\pi_{j}$. Define $x_{n}(\pi_{\text{rand}})$ similarly. Multiplying~\eqref{eq:2212} by $\rho_{n}$ for all classes $n$ and noting that vertices of the polytope $\Omega$ are performance vectors of strict priority policies, we have
\[
(x_{n}(\pi_{\text{rand}}))_{n=1}^{N} = \sum_{j\in\mathcal{J}} \alpha_{j} \, (x_{n}(\pi_{j}))_{n=1}^{N} \in \Omega,
\]
which proves the first part.

In the converse, for any given vector $(\oW_{n})_{n=1}^{N}$ in the delay region $\mathcal{W}$,  there exists a probability distribution $\{\beta_{j}\}_{j\in\mathcal{J}}$ such that
\[
\rho_{n} \oW_{n} = \sum_{j\in\mathcal{J}} \beta_{j}\, x_{n}(\pi_{j}) \Rightarrow
\oW_{n} =  \sum_{j\in\mathcal{J}} \beta_{j}\, \oW_{n}(\pi_{j})
\]
for all classes $n$. From~\eqref{eq:2212}, the randomized $\pi_{\text{rand}}$ policy defined by the probability distribution $\{\beta_{j}\}_{j\in\mathcal{J}}$ achieves the desired average delays $(\oW_{n})_{n=1}^{N}$.
\end{proof}

\section{} \label{appendix:701}

\begin{lem} \label{lem:608}
In a multi-class $M/G/1$ queue with $N$ classes and a constant service rate (assuming a constant power allocation and no power control), if the first four moments of service times $X_{n}$ are finite for all classes $n \in \{1, \ldots, N\}$, and that the system is stable with $\sum_{n=1}^{N} \lambda_{n}\expect{X_{n}}<1$, then, in every frame $k\in\Z^{+}$, the expectation
\[
\expect{\Big(\sum_{i\in A_{n,k}} \left(W_{n,k}^{(i)} - d_{n}\right)\Big)^{2}}
\]
is finite for all classes $n$ under any work-conserving policy.
\end{lem}

\begin{proof}[Proof of Lemma~\ref{lem:608}]
For brevity, we only give a sketch of proof. Using $\expect{(a-b)^{2}} \leq 2 \expect{a^{2}+b^{2}}$, it suffices to show
\[
\expect{\Big(\sum_{i\in A_{n,k}} W_{n,k}^{(i)}\Big)^{2}}, \quad \expect{ d_{n}^{2} \abs{A_{n,k}}^{2}}
\]
are both finite. We only show the first expectation is finite; the finiteness of the second expectation follows that of the first expectation. Define $N_{k}$ as the number of arrivals of all classes served in frame $k$; we have $\abs{A_{n,k}} \leq N_{k}$ for all $k$ and classes $n$. In the $k$th frame, since the queueing delay $W_{n,k}^{(i)}$ of each job $i\in A_{n,k}$ is bounded by the busy period $B_{k}$, we have
\[
\expect{\Big(\sum_{i\in A_{n,k}} W_{n,k}^{(i)}\Big)^{2}} \leq \expect{B_{k}^{2} N_{k}^{2}}.
\]
Note that $B_{k}$ and $N_{k}$ are dependent because a large busy period serves more jobs. By Cauchy-Schwarz inequality we have
\[
\expect{B_{k}^{2} N_{k}^{2}} \leq \sqrt{\expect{B_{k}^{4}}\expect{N_{k}^{4}}}.
\]
It suffices to show that both $\expect{B_{k}^{4}}$ and $\expect{N_{k}^{4}}$ are finite.

First we argue $\expect{B_{k}^{4}}<\infty$. In the following we drop the index $k$ for notational convenience. Since the frame size $B$ is the same under any work-conserving policy, we consider LIFO scheduling with preemptive priority. In this scheme, let $a_{0}$ denote the arrival that starts the current busy period. Arrival $a_{0}$ can be of any class, and the duration it stays in the system is equal to the busy period $B$. Next, let $\{a_{1}, \ldots, a_{M}\}$ denote the $M$ jobs that arrive \emph{during the service of job $a_{0}$}.  Let $B(1), \ldots, B(M)$ denote the duration they stay in the system. Under LIFO with preemptive priority, we observe that $B(1), \ldots, B(M)$ are independent and identically distributed with the starting busy period $B$ (since any new arrival \emph{never sees} any previous arrivals, and starts a new busy period). Consequently, we have
\begin{equation} \label{eq:1108}
B = X + \sum_{m=1}^{M} B(m),
\end{equation}
where $X$ denote the service time of $a_{0}$. Note also that each duration $B(m)$ for all $m\in\{1, \ldots, M\}$ is independent of $M$. By taking square and expectation of~\eqref{eq:1108}, we can compute $\expect{B^{2}}$ in closed form and show that it is finite if the first two moments of $X_{n}$ for all $n$ are finite. Likewise, by raising~\eqref{eq:1108} to the third and fourth power and taking expectation, we can compute $\expect{B^{3}}$ and $\expect{B^{4}}$ and show they are finite if the first four moments of $X_{n}$ are finite (showing $\expect{B^{4}}<\infty$ requires the finiteness of the first three moments of $B$).

Likewise, to show $\expect{N^{4}}$ is finite, under LIFO with preemptive priority we observe
\begin{equation} \label{eq:1109}
N = 1 + \sum_{m=1}^{M} N(m),
\end{equation}
where $N(m)$ denotes the number of arrivals, including $a_{m}$, served during the course of arrival $a_{m}$ staying in the system;  $N(m)$ are i.i.d. and independent of $M$. By raising~\eqref{eq:1109} to the second, third, and fourth power and taking expectation, we can compute $\expect{N^{4}}$ in closed form and show it is finite.
\end{proof}

\section{} \label{sec:2203}

\begin{proof}[Proof of Lemma~\ref{lem:1801}]
From~\eqref{eq:622} we get
\[
Y_{n,k+1} \geq Y_{n,k} - r_{n,k} \abs{A_{n,k}} + \sum_{i\in A_{n,k}} W_{n,k}^{(i)}.
\]
Summing over $k\in\{0, \ldots, K-1\}$ and using $Y_{n,0}=0$ yields
\[
\sum_{k=0}^{K-1} \sum_{i\in A_{n,k}} W_{n,k}^{(i)} - Y_{n,K} \leq \sum_{k=0}^{K-1} r_{n,k} \abs{A_{n,k}}.
\]
Taking expectation and dividing by $\lambda_{n}\, \expect{\sum_{k=0}^{K-1} T_{k}}$ yields
\begin{equation} \label{eq:2201}
\begin{split}
&\frac{ \expect{\sum_{k=0}^{K-1}  \sum_{i\in A_{n,k}} W_{n,k}^{(i)} } }{\lambda_{n} \expect{\sum_{k=0}^{K-1} T_{k}}} - \frac{\expect{Y_{n,K}}}{\lambda_{n} K \expect{T_{0}}} \\
&\quad \leq \frac{\expect{\sum_{k=0}^{K-1} r_{n,k} \abs{A_{n,k}}}}{\lambda_{n} \expect{\sum_{k=0}^{K-1} T_{k}}}.
\end{split}
\end{equation}
where in the second term we use $\expect{T_{k}} = \expect{T_{0}}$ for all $k$. In the last term of~\eqref{eq:2201}, since the value $r_{n,k}$ is  independent of $\abs{A_{n,k}}$ and $T_{k}$, we get
\begin{equation} \label{eq:2202}
\frac{\expect{\sum_{k=0}^{K-1} r_{n,k} \abs{A_{n,k}}}}{\lambda_{n} \expect{\sum_{k=0}^{K-1} T_{k}}} = \frac{\expect{\sum_{k=0}^{K-1} r_{n,k}\, T_{k}}}{\expect{\sum_{k=0}^{K-1} T_{k}}}.
\end{equation}
Defining as $\theta^{(n)}_{K}$ the left side of~\eqref{eq:2201}  and using~\eqref{eq:2201}~\eqref{eq:2202} yields
\begin{equation} \label{eq:1111}
\theta^{(n)}_{K} \leq \frac{\expect{\sum_{k=0}^{K-1} r_{n,k}\, T_{k}}}{\expect{\sum_{k=0}^{K-1} T_{k}}}.
\end{equation}
Since $f_{n}(\cdot)$ is nondecreasing for all classes $n$, we get
\begin{equation} \label{eq:2203}
\begin{split}
& \limsup_{K\to\infty} \sum_{n=1}^{N} f_{n}\left( \theta^{(n)}_{K} \right) \\
&\quad \leq \limsup_{K\to\infty} \sum_{n=1}^{N} f_{n}\left( \frac{\expect{\sum_{k=0}^{K-1} r_{n,k}\, T_{k}}}{\expect{\sum_{k=0}^{K-1} T_{k}}} \right).
\end{split}
\end{equation}
Define the value
\begin{equation} \label{eq:1802}
\eta^{(n)}_{K} \triangleq \frac{ \expect{\sum_{k=0}^{K-1}  \sum_{i\in A_{n,k}} W_{n,k}^{(i)} } }{ \expect{\sum_{k=0}^{K-1}\abs{A_{n,k}} }  } = \theta_{K}^{(n)} + \frac{\expect{Y_{n,K}}}{\lambda_{n} K \expect{T_{0}}}. 
\end{equation}
To complete the proof, from~\eqref{eq:2203} it suffices to show
\begin{equation} \label{eq:1114}
\limsup_{K\to\infty} \sum_{n=1}^{N} f_{n}(\eta^{(n)}_{K}) = \limsup_{K\to\infty} \sum_{n=1}^{N} f_{n}(\theta^{(n)}_{K}).
\end{equation}
Let the left-side of~\eqref{eq:1114} attains its $\limsup$ in the subsequence $\{K_{m}\}_{m=1}^{\infty}$. It follows
\begin{align*}
\limsup_{K\to\infty} \sum_{n=1}^{N} f_{n}(\eta^{(n)}_{K})
&= \lim_{m\to\infty} \sum_{n=1}^{N} f_{n}(\eta^{(n)}_{K_{m}}) \\
&\stackrel{(a)}{=}  \sum_{n=1}^{N} f_{n}\left(\lim_{m\to\infty} \eta^{(n)}_{K_{m}}\right) \\
&\stackrel{(b)}{=}  \sum_{n=1}^{N} f_{n}\left(\lim_{m\to\infty} \theta^{(n)}_{K_{m}}\right)  \\
&\leq \limsup_{K\to\infty} \sum_{n=1}^{N} f_{n}(\theta^{(n)}_{K}),
\end{align*}
where (a) follows the continuity of $f_{n}(\cdot)$ for all classes $n$, (b) follows~\eqref{eq:1802} and mean rate stability of $Y_{n,k}$. The other direction can be proved similarly.
\end{proof}

\section{} \label{sec:2201}

We show how to remove the dependence on the second moments of job sizes $S_{n}$  in the $\mathsf{DynPower}$ policy in Section~\ref{sec:1102}. Using~\eqref{eq:625}, we rewrite~\eqref{eq:607} as
\begin{equation} \label{eq:2206}
\begin{split}
& \hat{R} \Bigg[ \left( \frac{V}{\hat{R}} \sum_{n=1}^{N} \lambda_{n} \expect{S_{n}} \right) \frac{P_{k}}{\mu(P_{k})} \\
& + \sum_{n=1}^{N} \frac{Z_{\pi_{n},k} \, \lambda_{\pi_{n}}}{(\mu(P_{k}) - \sum_{m=0}^{n-1} \hat{\rho}_{\pi_{m}})(\mu(P_{k}) - \sum_{m=0}^{n} \hat{\rho}_{\pi_{m}})}  \Bigg]
\end{split}
\end{equation}
where 
\[
\hat{R} \triangleq \frac{1}{2} \sum_{n=1}^{N} \lambda_{n} \expect{S_{n}^{2}}, \ \hat{\rho}_{\pi_{m}} \triangleq \begin{cases}  \lambda_{\pi_{m}} \expect{S_{\pi_{m}}}, & 1\leq m \leq N \\ 0, & \text{otherwise.} \end{cases}
\]
By ignoring constant $\hat{R}$ and redefining $\wt{V} \triangleq V/\hat{R}$ in~\eqref{eq:2206}, it is equivalent in the $k$th frame of the $\mathsf{DynPower}$ policy to allocate power $P_{k}\in[\pmin, \pmax]$ that minimizes
\begin{equation} \label{eq:2207}
\begin{split}
& \left( \wt{V} \sum_{n=1}^{N} \lambda_{n} \expect{S_{n}} \right) \frac{P_{k}}{\mu(P_{k})} \\
&\quad + \sum_{n=1}^{N} \frac{Z_{\pi_{n},k} \, \lambda_{\pi_{n}}}{(\mu(P_{k}) - \sum_{m=0}^{n-1} \hat{\rho}_{\pi_{m}})(\mu(P_{k}) - \sum_{m=0}^{n} \hat{\rho}_{\pi_{m}})}.
\end{split}
\end{equation}
The sum~\eqref{eq:2207} does not depend on second moments of job sizes. From Theorem~\ref{thm:603} and using $V = \wt{V}\hat{R}$, this alternative policy yields average power $\Pav$ satisfying
\[
\Pav \leq \frac{C\sum_{n=1}^{N}\lambda_{n}}{\wt{V} \hat{R}} + P^{*},
\]
and we preserve the property that the resulting average $\Pav$ is $O(1/\wt{V})$ away from the optimal $P^{*}$.

\section{} \label{sec:2401}

\begin{proof}[Proof of Lemma~\ref{lem:1701}]
From~\eqref{eq:1702} we have
\[
X_{k+1} \geq X_{k} + P_{k} B_{k}(P_{k}) - \pconst \, T_{k}(P_{k}).
\]
Summing over $k\in\{0, \ldots, K-1\}$, taking expectation, and using $X_{0}=0$ yields
\[
\expect{X_{K}} \geq \expect{\sum_{k=0}^{K-1} P_{k} B_{k}(P_{k})} - \pconst \, \expect{\sum_{k=0}^{K-1} T_{k}(P_{k})}.
\]
Dividing by $\expect{\sum_{k=0}^{K-1} T_{k}(P_{k})}$ and passing $K\to\infty$ yields
\begin{align*}
\Pav &\leq \pconst + \limsup_{K\to\infty} \frac{\expect{X_{K}}}{K} \frac{K}{\expect{\sum_{k=0}^{K-1} T_{k}(P_{k})}} \\
&\stackrel{(a)}{\leq} \pconst + \limsup_{K\to\infty} \frac{\expect{X_{K}}}{K} \sum_{n=1}^{N} \lambda_{n},
\end{align*}
where (a) uses $\expect{T_{k}(P_{k})} \geq \expect{I_{k}} = 1/ (\sum_{n=1}^{N} \lambda_{n})$.
Then the result follows by mean rate stability of queue $\{X_{k}\}_{k=0}^{\infty}$.
\end{proof}

\end{document}